\documentclass[11pt]{amsart}
\usepackage{amsfonts} 
\usepackage{amssymb}
\usepackage{amsthm}
\usepackage{amsmath} 
\usepackage{mathrsfs}
\usepackage{mathtools}
\usepackage{dsfont}
\usepackage{esint}
\usepackage{ulem}
\usepackage{marginnote}
\usepackage{array}
\usepackage{algorithm}
\usepackage{algpseudocode}
\usepackage{soul}
\usepackage{comment}

\usepackage{booktabs} 

\usepackage{MnSymbol}

\usepackage{longtable}

\usepackage[all]{xy}

\usepackage[english]{babel} 
\usepackage[left=3.5cm,right=3.5cm,top=3.5cm,bottom=3cm,headsep=0.7cm]{geometry}
\usepackage{xargs}

\usepackage{csquotes}
\usepackage{stmaryrd}

\usepackage{pgf,tikz}
\usetikzlibrary{positioning, calc, arrows.meta}
\usetikzlibrary{arrows}
\usepackage{subcaption}
\usepackage{graphicx}
\usepackage[parfill]{parskip}
\usepackage{tikz-cd}

\usepackage{enumerate}
\usepackage{enumitem}

\usepackage{hyperref}
\usepackage[nameinlink]{cleveref}
\usepackage{mathtools}
\hypersetup{
	colorlinks,
	linkcolor={blue!80!black},
	citecolor={blue!80!black}
}

\theoremstyle{definition}
\newtheorem{definition}{Definition}[section]

\theoremstyle{plain}
\newtheorem{theorem}[definition]{Theorem}
\newtheorem{proposition}[definition]{Proposition}
\newtheorem{lemma}[definition]{Lemma}

\theoremstyle{remark}
\newtheorem{remark}[definition]{Remark}

\newcommand{\N}{\mathbb N}

\newcommand{\R}{\mathbb R}
\newcommand{\C}{\mathbb C}
\newcommand{\B}{\N^{\N}}
\newcommand{\Cantor}{2^{\N}}
\newcommand{\dom}{\operatorname{dom}}
\newcommand{\id}{\operatorname{id}}
\newcommand{\Comp}{\mathrm{Comp}}
\newcommand{\SCIG}{\operatorname{SCI}_{\mathrm G}}
\newcommand{\SCImv}{\operatorname{SCI}^{\mathrm{mv}}_{\mathrm G}}
\newcommand{\puremv}{\operatorname{SCI}^{\lim,\mathrm{mv}}_{\Comp}}
\newcommand{\pure}{\operatorname{SCI}^{\lim}_{\Comp}}
\newcommand{\rk}{\operatorname{rk}_{\lim}}
\newcommand{\Lim}{\operatorname{lim}}
\newcommand{\lW}{\leq_{\mathrm W}}
\newcommand{\sW}{\leq_{\mathrm{sW}}}
\newcommand{\eW}{\equiv_{\mathrm W}}
\newcommand{\esW}{\equiv_{\mathrm{sW}}}
\newcommand{\nW}{\nleq_{\mathrm W}}
\newcommand{\LW}{<_{\mathrm W}}
\newcommand{\LPO}{\operatorname{LPO}}
\newcommand{\INF}{\operatorname{INF}}
\newcommand{\FIN}{\operatorname{FIN}}
\newcommand{\IVT}{\operatorname{IVT}}
\newcommand{\CC}{\operatorname{CC}_{[0,1]}}
\newcommand{\BWT}{\operatorname{BWT}_{2}}
\newcommand{\Solve}{\operatorname{Solve}}
\newcommand{\Ev}{\operatorname{Ev}}
\newcommand{\Names}{\mathcal N}
\newcommand{\one}{\mathbf 1}
\newcommand{\pair}[1]{\langle #1\rangle}
\newcommand{\norm}[1]{\lVert #1\rVert}
\newcommand{\abs}[1]{\lvert #1\rvert}
\newcommand{\pos}[1]{(#1)_{+}}

\title{\bf From raw Solvability Complexity Index proofs to Weihrauch degrees}

\begin{document}

\author[C.~Sorg]{Christopher Sorg$^1$}
\address[C.~Sorg]{
	\textup{Chair for Theoretical computer science, mathematics, and operations research}
	\newline \indent
	\textup{Department of Computer Science} \newline \indent
	\textup{University of the Bundeswehr Munich}
	\newline \indent
	\textup{85577 Neubiberg, Germany}}

\email{{\href{mailto:chr.sorg@unibw.de}{\textcolor{blue}{\texttt{chr.sorg@unibw.de}}}}
}

\footnotetext[1]{Inf1, University of the Bundeswehr Munich, Werner-Heisenberg-Weg 39, 85577 Neubiberg, Germany}

\begin{abstract}
The Solvability Complexity Index (SCI) provides an extensional limit-height formalism for recovering a target map $\Xi$ from finite samples of an evaluation interface $\Lambda$ by finite-height towers of pointwise limits. At first sight this sounds like a typical Type-2 question, so it seems natural to ask how deep and rich the connection of SCI approximation questions and Type-2 computability questions is. In this \textbf{note} I want to shape the picture my logical studies in \cite{SorgBridge, SorgWitness, SorgTransport} gave me to this day. I am trying to do that by giving explicit examples - most of them directly from the literature, in which case I extracted the crucial core proof patterns - as many people directly or implicitly expressed this wish. I am always open for discussions, questions and remarks about my work, so don't hesitate to contact me in these cases.
\end{abstract}

\maketitle
\tableofcontents
\clearpage

\begin{figure}[!h]
	\centering
	\begin{tikzpicture}[
		x=1cm,y=1cm,
		box/.style={draw,rounded corners=2pt,align=center,inner sep=7pt,
			font=\sffamily\small,text width=6.65cm,minimum height=1.15cm},
		wide/.style={box,text width=14.4cm},
		ann/.style={align=center,font=\sffamily\scriptsize,text width=6.1cm},
		arr/.style={-{Latex[length=2mm]},line width=.55pt}
		]
		\node[wide,minimum height=1.65cm] (tte) at (0,0) {
			\textbf{TRACED TTE FINITE-QUERY TRANSPORT}\\[4pt]
			$S\leq^{\mathrm{mv}}_{\mathrm{TTE},\mathrm{fq}}P$\\[3pt]
			Computable point maps $E,D$ and a uniform finite interface trace;\\
			$D$ is total and metrically continuous; solution inclusion holds.\\
			{\scriptsize [Definition 4.2 here]}
		};
		\node[box] (raw) at (-3.85,-4.00) {
			\textbf{RAW FINITE-QUERY TRANSPORT}\\[5pt]
			$S\leq^{\mathrm{mv}}_{\mathrm{raw},\mathrm{fq}}P$
		};
		\node[box] (strong) at (3.85,-4.00) {
			\textbf{STRONG WEIHRAUCH COMPARISON}\\[5pt]
			$\widehat F_S\leq_{\mathrm{sW}}\widehat F_P$
		};
		\draw[arr] ($(tte.south)+(-7.15,0)$) -- ($(raw.north)+(-3.3,0)$);
		\draw[arr] ($(tte.south)+(0.55,0)$) -- ($(strong.north)+(-3.3,0)$);
		\node[ann] at (-3.65,-2.55) {Forget computability\\[2pt]
			[Definition 4.2 here]};
		\node[ann] at (4.05,-2.55) {Forget the transport data\\[2pt]
			[Theorem 4.3 here; S26-T, Thm. 5.1\\
			for the single-valued source statement]\\
			Converse fails: Proposition 4.4 here};
		\node[box] (rawrank) at (-3.85,-7.30) {
			\textbf{RAW-HEIGHT COMPARISON}\\[5pt]
			$\operatorname{SCI}^{\mathrm{mv}}_G(S)
			\leq\operatorname{SCI}^{\mathrm{mv}}_G(P)$
		};
		\node[box] (ordinary) at (3.85,-7.30) {
			\textbf{ORDINARY WEIHRAUCH COMPARISON}\\[5pt]
			$[\widehat F_S]_{\mathrm W}\leq[\widehat F_P]_{\mathrm W}$
		};
		\draw[arr] ($(raw.south)+(-3.3,0)$) -- ($(rawrank.north)+(-3.3,0)$);
		\draw[arr] ($(strong.south)+(-3.3,0)$) -- ($(ordinary.north)+(-3.3,0)$);
		\node[ann] at (-3.65,-5.70) {Pull back every raw tower\\[2pt]
			[Theorem 4.3 here; S26-W, Thm. 4.10\\
			for the single-valued source statement]};
		\node[ann] at (4.05,-5.70) {Allow the original input name\\
			in the postprocessor\\[2pt]
			[Section 1.1 here]};
		\node[wide,minimum height=1.6cm] (rank) at (0,-10.30) {
			\textbf{ORDER-PRESERVING RANK PROJECTION}\\[4pt]
			$[f]_{\mathrm W}\longmapsto
			\operatorname{rk}_{\lim}(f)
			=\min\{k\in\mathbb N:f\leq_{\mathrm W}\lim\nolimits^{(k)}\}$\\[4pt]
			$\operatorname{rk}_{\lim}(\widehat F_S)
			\leq\operatorname{rk}_{\lim}(\widehat F_P)$\\[3pt]
			{\scriptsize [Theorem 4.3 / (15) here; S26-F, Def. 4.11 and Thm. 4.12]}\\
			{\scriptsize Non-injectivity: Theorem 3.2 here.}
		};
		\draw[arr] (ordinary.south) -- ($(rank.north)+(3.85,0)$)
		node[midway,right,font=\sffamily\scriptsize,xshift=3pt] {Take the least bound};
		\node[wide,minimum height=2.0cm] (pure) at (0,-13.60) {
			\textbf{FOR EACH FIXED REPRESENTED TARGET $f=\widehat F_R$}\\[5pt]
			$\displaystyle
			\operatorname{SCI}^{\lim,\mathrm{mv}}_{\mathrm{Comp}}(R)
			=\min\bigl\{k\in\mathbb N:\exists K\in\mathrm{Comp},\
			\lim\nolimits^{(k)}\!\circ K\preceq\mathcal N_f\bigr\}
			=\operatorname{rk}_{\lim}(f)$\\[6pt]
			{\scriptsize Normal form: S26-F, Thm. 4.27. Single-valued minimum: Theorem 1.2 here.}\\
			{\scriptsize Multivalued minimum: Definition 4.1 / (10) here.}\\
			{\scriptsize Equality of two ranks, not an identification of exact degrees.}
		};
		\node[wide,dashed,minimum height=2.0cm] (impl) at (0,-17.10) {
			\textbf{ADDITIONAL IMPLEMENTATION WORK - NOT A RAW-ORDER IMPLICATION}\\[5pt]
			For the fixed target $f$, construct one computable $K$ and prove\\
			$\lim\nolimits^{(m)}\!\circ K\preceq\mathcal N_f$,
			including every iterated-limit domain condition.\\[4pt]
			This proves a pure-height upper bound $\leq m$; sharpness needs a lower bound.\\[4pt]
			{\scriptsize Worked implementations: Theorems 2.4, 3.2 and 9.1 here.}\\
			{\scriptsize No map from raw classes to pure heights: Proposition 7.4 and Theorem 9.1 here.}
		};
	\end{tikzpicture}
	\caption{Illustration of the meta-idea of the notes. The dashed box is a proof obligation, not an implication or a map from raw towers or raw degrees.}
	\label{fig:rank-trace-map}
\end{figure}
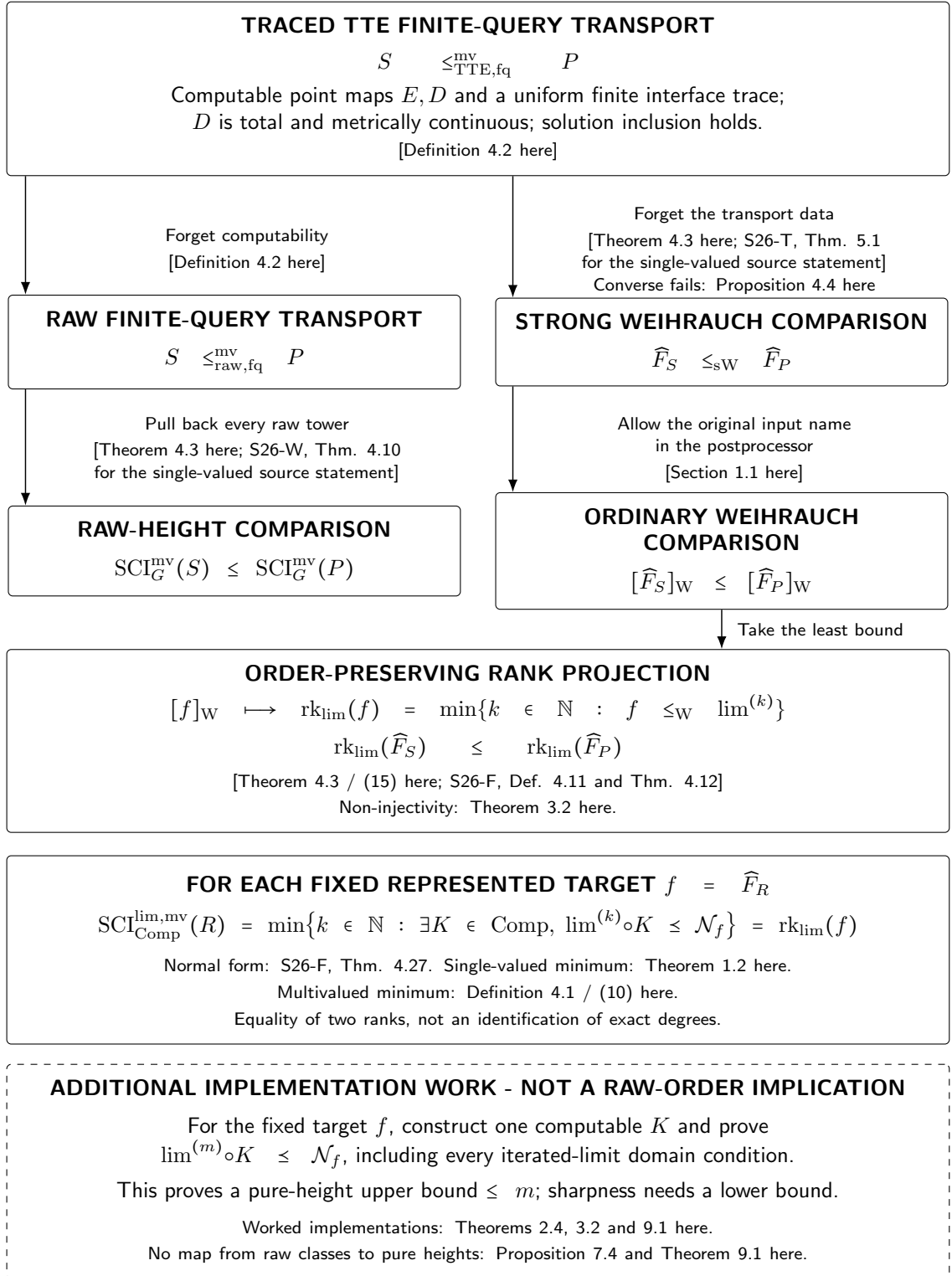

\clearpage

\part{Proof-level translations of literature classifications}\label{part:literature}
\section{Given setting and needed additional results}\label{sec:bridge}

\subsection{Raw SCI, representations, and exact degrees}

Throughout, $\N=\{0,1,2,\ldots\}$; this merely reindexes the positive-integer conventions in some sources. Fix computable codings of finite tuples, of $\N^2$ by $\N$, and of rational numbers by natural numbers.

\begin{definition}[Raw framework: {\cite[Definitions 2.4-2.9, 2.17-2.18]{SorgBridge}}]\label{def:raw}
An SCI problem is a quadruple
\[
 P=(\Xi,\Omega,(M,d),\Lambda),\qquad \Xi:\Omega\to M,\quad
 \Lambda\subseteq\C^\Omega,
\]
satisfying
\[
 \bigl(\forall\lambda\in\Lambda\;\lambda(A)=\lambda(B)\bigr)
 \ \Longrightarrow\ \Xi(A)=\Xi(B).
\]
A general algorithm is a pair $(\Gamma,Q_\Gamma)$ with $\Gamma:\Omega\to M$ and a finite nonempty set $Q_\Gamma(A)\subseteq\Lambda$ for each $A$, such that agreement with $A$ on $Q_\Gamma(A)$ implies both
\[
 \Gamma(B)=\Gamma(A),\qquad Q_\Gamma(B)=Q_\Gamma(A).
\]
A height-$k$ raw tower consists of deepest maps $\Gamma_{n_k,\ldots,n_1}$ that are general algorithms, with all successive pointwise limits existing in $M$ and
\[
 \Xi(A)=\lim_{n_k\to\infty}\cdots\lim_{n_1\to\infty} \Gamma_{n_k,\ldots,n_1}(A).
\]
The intermediate limit maps need \textit{not} themselves be general algorithms. Height $0$ means that $\Xi$ itself is a general algorithm. The least height is $\SCIG(P)$.
\end{definition}

This is the deepest-level-only convention of \cite[Definitions 2.4, 2.5, 2.9, 2.17 and 2.18]{SorgBridge}. A harmless fixed dummy query supplies nonemptiness whenever an algorithm otherwise uses no queries. Crucially, raw type~G places \textit{no computability or continuity restriction} on the function of the finitely many exact complex values that is used to produce the output.

We assume in this work raw SCI interfaces to be nonempty. If no finite-height raw tower exists, set $\SCIG(P)=\infty$.

Whenever a construction would otherwise use no queries, it queries and ignores a fixed evaluation $\lambda_\circ\in\Lambda$. This padding preserves the given interface.

A representation of a set $X$ is a partial surjection $\delta_X:\subseteq\B\to X$. We use the usual real Cauchy representation $\rho$: if $r_j$ is the rational coded by $p(j)$, then
\[
 \rho(p)=x\quad\Longleftrightarrow\quad
 \forall j\;\abs{r_j-x}\leq 2^{-j}.
\]
A name of a sequence of reals gives such names uniformly in the sequence index. Cantor space has its identity representation, and $2=\{0,1\}$ has the discrete representation $\delta_2(p)=p(0)$ when $p(0)\in2$.

If $\Lambda=(\lambda_i)_{i\in\N}$ is countable, define
\[
 \Ev_\Lambda(A)=(\lambda_i(A))_{i\in\N},\qquad
 I_\Lambda=\Ev_\Lambda(\Omega)\subseteq\C^\N.
\]
The work in \cite{SorgBridge} equips $I_\Lambda$ with the \textit{subspace product Cauchy representation}. The consistency condition gives a well-defined factor
\[
 \widehat\Xi:I_\Lambda\to M,\qquad
 \widehat\Xi(\Ev_\Lambda(A))=\Xi(A).
\]
An output representation $\delta_M$ is additional data. A countable enumeration is itself effective data: invariance is under \textit{computable} reindexings, not arbitrary permutations presented without effective access.

For represented problems $f:\subseteq X\rightrightarrows Y$ and $g:\subseteq U\rightrightarrows V$, write $f\lW g$ if computable name-level functionals $K,H$ satisfy
\[
 p\longmapsto H\bigl(p,G(K(p))\bigr)\quad\text{realizes }f
\]
for \textit{every} realizer $G$ of $g$. For strong reducibility $f\sW g$, the postprocessor is $H(G(K(p)))$, without the original input. The degree $\one$ denotes the degree of computable problems with a computable instance; it is not the empty-domain bottom degree.

\subsection{The bridge is a rank theorem, not a degree theorem}

For $P\in\B$, regard $P(\pair{s,j})$ as an array and define
\[
 \Lim(P)(j)=\lim_{s\to\infty}P(\pair{s,j}),
\]
with domain consisting of arrays for which each column is eventually constant. Put
\[
 \Lim^{(0)}=\id_\B,\qquad
 \Lim^{(k+1)}=\Lim\circ\Lim^{(k)},\qquad
 \rk(f)=\min\{k:f\lW\Lim^{(k)}\}.
\]
The minimum is $\infty$ when no finite $k$ is possible. These are \textit{compositions} of limit operators; the derivative notation $f'$ introduced below has a different definition.

For a represented problem $f$, let
\[
 \Names_f(p)=\{q:\delta_Y(q)\in f(\delta_X(p))\}
\]
on names of inputs in $\dom(f)$. A single-valued partial functional $F$ tightens $\Names_f$, written $F\preceq\Names_f$, if it is defined on all these names and $F(p)\in\Names_f(p)$ there.

\begin{theorem}[Source normal form and specialization: {\cite[Theorems 4.27 and 8.25]{SorgBridge}}]\label{thm:bridge}
For $\mathcal R=\Comp$, the pure implemented height is
\[
 \pure(P)=\min\bigl\{k:\exists\text{ computable }K\quad \Lim^{(k)}\circ K\preceq\Names_{\widehat\Xi}\bigr\}.
\]
With fixed representations as above,
\begin{equation}\label{eq:bridge}
 \pure(P)=\rk(\widehat\Xi).
\end{equation}
More generally, the same statement holds for a class $\mathcal R$ when its reduction notion is defined and the corresponding $\mathcal R$-normal-form property holds. No such property is presumed for arbitrary classes.
\end{theorem}

\begin{proof}
The computable normal form is
\begin{equation}\label{eq:nf}
 f\lW\Lim^{(k)}\quad\Longleftrightarrow\quad
 \exists\text{ computable }K\; \Lim^{(k)}\circ K\preceq\Names_f.
\end{equation}
This is \cite[Theorem 4.27]{SorgBridge}; taking the least $k$ gives \eqref{eq:bridge}, which is \cite[Definition 8.24 and Theorem 8.25]{SorgBridge}. Its crucial content for an application is the existence of \textit{one} preprocessor $K$ producing the whole array, not separate computability claims for individual estimators.
\end{proof}

The conditions denoted P1-P5 in \cite[Theorem 8.26]{SorgBridge} are implemented here as follows: fixed represented input and output spaces; Cauchy names rather than computationally atomic exact real values; computable base functionals; the usual closure of these functionals under identities, pairing, projections, computable constants and composition; and \eqref{eq:nf}. We use this \textit{sufficient concrete package}.

Three different conclusions must not be conflated:
\[
 \boxed{\SCIG(P)=k}\qquad
 \boxed{\pure(P)=\rk(\widehat\Xi)=k}\qquad
 \boxed{\widehat\Xi\eW D}.
\]
The first box does not generally imply the second. Even the second only says
\[
 \widehat\Xi\lW\Lim^{(k)},\qquad
 \widehat\Xi\nW\Lim^{(k-1)}\quad(k>0),
\]
not $\widehat\Xi\eW\Lim^{(k)}$. An exact identification with $D$ requires \textit{both} $\widehat\Xi\lW D$ and $D\lW\widehat\Xi$.

\begin{remark}[What ``depends on the choices'' means]\label{rem:choices}
Changing a representation within its computable equivalence class does not change an ordinary Weihrauch degree. Changing the base class $\mathcal R$ changes the reduction notion, not the already fixed ordinary degree of a represented problem. Thus there are not several answers for one fully fixed represented problem. There are different represented problems or different reduction notions. The IVT examples below make these changes explicit.
\end{remark}

\subsection{Is this exactly the SCI-Weihrauch intermediate hierarchy?}\label{sec:intermediate}
There is a relevant intermediate hierarchy, but three different mathematical structures must be kept separate.

\textbf{The topological hierarchy from \cite{SorgBridge}.}
For a topologized SCI problem $P_\tau$, \cite[Definitions 6.12-6.14 and Theorem 6.16]{SorgBridge} restricts the deepest general algorithms by a regularity label $r$, and organizes
\[
 \mathcal H_{(r,k)}=\{P_\tau:\operatorname{SCI}_{\mathrm G,r}(P_\tau)\leq k\}.
\]
For example $r=\mathrm{Cont}$ or $r=\mathrm{Bor}$ requires continuity or Borel measurability of the relevant object-level base maps. The optional query-policy parameter $q\in\{\mathrm{fix},\mathrm{ad}\}$ is introduced in \cite[Definitions 6.17-6.18 and Proposition 6.19]{SorgBridge}. This is the concrete regularity/height ``intermediate hierarchy'' in Section~6. Its relation to descriptive complexity is developed in Theorems~6.25-6.26 of that source: continuous deepest maps imply the corresponding Baire-class upper bound, and, in the standard single-valued represented-Polish setting stated there, continuous Weihrauch reduction to $\Lim^{(k)}$ characterizes the corresponding Borel-measurability level. The effective characterization used there is \cite[Theorem 6.5]{BGP}.

These are not unrestricted claims about arbitrary representations or arbitrary topological spaces. In particular, a continuous map need not have a computable realizer, and separate realizers for each indexed estimator need not be uniformly obtainable. Membership in the topological intermediate hierarchy therefore does not automatically give the computable preprocessor in \eqref{eq:nf}.

\textbf{The effective bridge is a name-level refinement.}
Fixing representations and $\mathcal R=\Comp$, define the certificate condition
\[
 \mathcal C_k(f):\quad \exists K\in\Comp\quad \Lim^{(k)}\circ K\preceq\Names_f.
\]
By \cite[Theorem 4.27]{SorgBridge}, this is equivalent to $f\lW\Lim^{(k)}$. Definition~8.24 adds this globally uniform implementation requirement; Theorem~8.25 identifies its minimum height with $\rk(f)$. Thus the bridge gives an \textit{exact correspondence between two rank descriptions after the effective choices and normal form have been fixed}. It does not identify the entire topological hierarchy with ordinary Weihrauch degrees. We keep $\mathcal R=\Comp$ throughout our degree calculations: replacing it by continuous or Borel functionals would change the reduction notion.

\textbf{An exact degree contains more information than this rank.}
There is a well-defined order-preserving projection
\[
 [f]_{\mathrm W}\longmapsto \min\{k:f\lW\Lim^{(k)}\},
\]
as in \cite[Definition 4.11 and Theorem 4.12]{SorgBridge}. It is many-to-one. In this note, the single-valued problems $\INF$ and $\Lim^{(2)}$ both have rank $2$ but are not equivalent (\Cref{thm:inf-degree}). Conversely, the unique-root IVT problems in \Cref{thm:literal} all have the same raw SCI value while their represented ranks and degrees differ. Hence restricting to single-valued problems does not repair either failed identification.

\textbf{Why single-valuedness matters, and where it does not.}
The raw definition in \cite[Definition 2.4]{SorgBridge}, and its pure package in Definition~8.24, have single-valued $\Xi$. The ordinary Weihrauch lattice and the normal form in Theorem~4.27 also accommodate multivalued problems. Although $\Lim^{(k)}$ is single-valued on names, its output can be a name of \textit{one permissible answer} to a multivalued problem; this does not turn the underlying relation into a single-valued map. We explicitly extend the raw and pure definitions when discussing ordinary IVT in \Cref{def:mv}. The two literature problems and the unique-root companion already fit the literal single-valued definition.

\textbf{Scope of the axiomatic claims used.}
Our arguments require the computable normal form and the definitions of the two minima, not the full minimality claim in \cite[Theorem 8.26]{SorgBridge}. In particular, ``constants'' in our computable package means \textit{computable} constant Baire-space maps. Taken literally, a demand for \textit{all} constant Baire-space maps, as in that source's Definition~8.12, is not satisfied by $\Comp$: a map constantly equal to a noncomputable sequence is not computable. This harmless restriction for the present constructions is stated explicitly rather than silently importing the stronger adequacy language. No closure assertion about noncomputable auxiliary classes is needed here.

\subsection{Three implementation lemmas}

\begin{lemma}[Explicit real-limit compiler supplied here]\label{lem:real-limit}
Given uniformly computable rational numbers $(x_n)_{n\in\N}$ that converge to a real $x$, with no supplied convergence rate, one can uniformly compute an array $A(s,k)$ of rational codes such that
\[
 \Lim(A)\text{ is a }\rho\text{-name of }x.
\]
The computation may be relative to an arbitrary input name, using the same machine for all $s,k$ and all inputs satisfying the convergence promise.
\end{lemma}

\begin{proof}
For $s,k\in\N$, put $\varepsilon_k=2^{-k-2}$ and
\[
 N_{k,s}=\min\bigl\{n\leq s: \forall m\in\{n,\ldots,s\}\;\abs{x_n-x_m}\leq\varepsilon_k\bigr\}.
\]
The set is nonempty because $n=s$ qualifies, and its finite rational tests are decidable. There is at least one $n$ satisfying the same inequalities for \textit{all} $m\geq n$, by the Cauchy property. Let $N_k$ be the least such $n$. Every $n<N_k$ has a finite witness to its failure. Consequently $N_{k,s}=N_k$ for all sufficiently large $s$.

Let $A(s,k)$ be the code of $x_{N_{k,s}}$. Each column stabilizes to the code of $x_{N_k}$, and letting $m\to\infty$ gives $\abs{x_{N_k}-x}\leq\varepsilon_k\leq2^{-k}$. This proves the claim, including the domain condition for $\Lim$. Notice that we take a limit of \textit{codes that stabilize}, not just a limit of arbitrary Cauchy names.
\end{proof}

Define $\LPO(u)=1$ if some $u(t)=1$, and $\LPO(u)=0$ otherwise, for $u\in\Cantor$. Its parallelization is
\[
 \widehat{\LPO}(p)(i)=1\quad\Longleftrightarrow\quad \exists t\;p(i,t)=1.
\]

\begin{lemma}[Standard identity: {\cite[Theorem 6.7]{BGP}}; proof included]\label{lem:plpo}
$\widehat{\LPO}\eW\Lim$, and this degree is noncomputable.
\end{lemma}
\begin{proof}
For the forward reduction, the finite-stage disjunctions $\max_{t\leq s}p(i,t)$ stabilize to the desired bits. Conversely, let $a_s(j)$ be an eventually constant natural-valued array. For each $(j,N)$ ask, in parallel, whether
\[
 \exists s\geq N\quad a_s(j)\neq a_N(j).
\]
Each is an LPO question, uniformly in the input array. For each $j$, search the answer table for an $N$ with answer $0$; one exists. Output $a_N(j)$. Noncomputability follows already from LPO: on the all-zero input, a finite computation cannot distinguish a later, unseen $1$. Equivalently, a computable array of halting-at-stage indicators has a noncomputable $\Lim$-output.
\end{proof}

For a representation $\delta$, define its jump by $\delta'=\delta\circ\Lim$ and iteratively $\delta^{(k)}=\delta\circ\Lim^{(k)}$. The derivative $f'$ is the \textit{same underlying relation} as $f$, but with the input representation replaced by its jump and its output representation left unchanged.

\begin{lemma}[Jump lifting and strong monotonicity: {\cite[Proposition 5.6]{BGM}}]\label{lem:lift}
A computable map between represented spaces remains computable when both representations are jumped. The construction iterates for every fixed finite number of jumps. Moreover,
\[
 f\sW g\quad\Longrightarrow\quad f'\sW g'.
\]
\end{lemma}
\begin{proof}
Let a partial computable functional $F$ transform names, and suppose $p_s\to p$ coordinatewise with $p\in\dom(F)$. At stage $s$, for each output coordinate $j$, simulate the computation of $F(p_s)(j)$ for $s$ steps, and output the resulting natural number if it has halted, or $0$ otherwise. This defines a total computable array transformer on the arrays $p_s$.

For a fixed $j$, the true computation $F(p)(j)$ uses finitely many oracle values and finitely many steps. Eventually $p_s$ agrees with $p$ on all queried positions and the simulation budget is sufficient. The output coordinate therefore stabilizes to $F(p)(j)$. This proves the one-jump claim, even though intermediate $p_s$ need not be valid names for the original representation. The array transformer just constructed is total on arrays, so applying the same construction to it gives the iterated version. For a strong reduction, lift its preprocessor this way and leave its ordinary output postprocessor unchanged.
\end{proof}

\section{Literature example 1: solving infinite linear systems}\label{sec:lin}

The first example can be found in \cite[Theorem 5.2, equation (5.3); Section 12, Step III]{BAHNS}. Its target is the solution vector $A^{-1}b$. We specialize its supplied dispersion bound to a fixed computable function, and reconstruct the finite-section convergence argument below. The lower-bound sentence in Step III is expanded into a finite-query proof.

\subsection{The source problem and the representation fixed for comparison}
For this section use $H=\ell^2(\N_{\geq1};\C)$ with standard basis $(e_j)_{j\geq1}$, norm $\norm{\cdot}$, and orthogonal projections $P_m$ onto $\operatorname{span}\{e_1,\ldots,e_m\}$. Fix a \textit{computable} function $f:\N_{\geq1}\to\N_{\geq1}$ with $f(m)\geq m$; $f(m)=2m$ is a concrete choice. Following \cite[Definition 3.5]{BAHNS}, put
\[
 D_{f,m}(A)=\max\{\norm{(I-P_{f(m)})AP_m},\, \norm{P_mA(I-P_{f(m)})}\}.
\]
The promise that $A$ has dispersion bounded by $f$ means precisely $D_{f,m}(A)\to0$. It does \textit{not} supply a convergence rate. Define
\[
 \begin{split}
 \Omega_f&=\{A\in\mathcal B(H):A\text{ is invertible and }D_{f,m}(A)\to0\}\times H,\\
 \Xi_f(A,b)&=A^{-1}b,\qquad M=H,\quad d(x,y)=\norm{x-y}.
 \end{split}
\]
Here invertibility means a bounded everywhere-defined inverse. Take the effectively enumerated, tagged evaluation family
\[
 \lambda^A_{ij}(A,b)=\langle Ae_j,e_i\rangle\quad(i,j\geq1),\qquad
 \lambda^b_i(A,b)=\langle b,e_i\rangle\quad(i\geq1).
\]
One may add the constant evaluations $\lambda^f_m(A,b)=f(m)$ used in \cite[Remark 5.1]{BAHNS}. Because $f$ is fixed and computable, this addition changes neither raw SCI nor the represented degree. Matrix and vector coordinates determine $(A,b)$, so the consistency condition holds.

The represented problem, denoted $\Solve_f$, receives \textit{product Cauchy names of these individual coordinates}, restricted to the promised domain. This is not an operator-action representation of $A$, nor a norm-Cauchy representation of $b$. No bounds on $\norm{A}$, $\norm{A^{-1}}$, or $\norm{b}$, and no dispersion or tail moduli, are included. The output has the norm-Cauchy representation $\delta_H$: fix an effective enumeration $(v_j)$ of finitely supported rational complex vectors, and set
\[
 \delta_H(q)=x\quad\Longleftrightarrow\quad \forall k\in\N\;\norm{v_{q(k)}-x}\leq2^{-k}.
\]
This distinction between coordinate input and norm output is essential to the degree.

\begin{proposition}[Source reconstruction: {\cite[Theorem 5.2, (5.3)]{BAHNS}}]\label{prop:lin-raw}
For the fixed computable $f$ above, $\SCIG(P_f)=1$.
\end{proposition}
\begin{proof}[Reconstruction of Section 12, Step III, with the lower bound expanded]
All finite-dimensional inverses below act on $P_mH$. Define
\[
 T_m=P_{f(m)}AP_m,\qquad C_m=T_m^*T_m,\qquad
 d_m=T_m^*P_{f(m)}b,
\]
and, writing $\lambda_{\min}$ for the least eigenvalue of a Hermitian matrix, set
\begin{equation}\label{eq:lin-raw}
 \Gamma_m(A,b)=
 \begin{cases}
 0,&\lambda_{\min}(C_m)\leq 1/m,\\
 C_m^{-1}d_m,&\lambda_{\min}(C_m)>1/m.
 \end{cases}
\end{equation}
This is the source's thresholded finite-section construction on the diagonal $n=f(m)$. Its fixed, finite query set consists of the first $f(m)$ rows and $m$ columns of $A$, and the first $f(m)$ coordinates of $b$. Exact tests and finite-dimensional inversion are permitted in raw type~G. The output is embedded in $H$ by zero extension.

Fix $(A,b)\in\Omega_f$ and write
\[
 x=A^{-1}b,\qquad \alpha=\norm{A^{-1}}^{-1}>0,\qquad
 E_m=(I-P_{f(m)})AP_m,\quad \eta_m=\norm{E_m}\longrightarrow0.
\]
For comparison introduce the \textit{untruncated} normal equations
\[
 B_m=P_mA^*AP_m,\qquad c_m=P_mA^*b,\qquad y_m=B_m^{-1}c_m.
\]
These are mathematical comparison objects, not finitely queried data. Since $B_m\geq\alpha^2 I_m$, $y_m$ is well defined and minimizes $\norm{Ay-b}$ over $y\in P_mH$. Therefore
\begin{equation}\label{eq:lin-ls}
 \norm{y_m-x}\leq\alpha^{-1}\norm{Ay_m-b}
 \leq\alpha^{-1}\norm{A}\norm{P_mx-x}\longrightarrow0.
\end{equation}
For $u\in P_mH$, $\norm{T_mu}\geq(\alpha-\eta_m)\norm{u}$. Thus, eventually,
\[
 C_m\geq\tfrac14\alpha^2 I_m,\qquad
 \norm{C_m^{-1}}\leq4\alpha^{-2},\qquad
 \lambda_{\min}(C_m)>1/m.
\]
Furthermore,
\[
 B_m-C_m=E_m^*E_m,\qquad c_m-d_m=E_m^*b,
\]
so for all sufficiently large $m$, with $z_m=C_m^{-1}d_m$,
\begin{equation}\label{eq:lin-perturb}
 \norm{z_m-y_m}
 \leq4\alpha^{-2}\bigl(\eta_m\norm{b}+\eta_m^2\norm{y_m}\bigr)
 \longrightarrow0.
\end{equation}
Equations \eqref{eq:lin-ls}-\eqref{eq:lin-perturb} prove $\Gamma_m(A,b)\to A^{-1}b$. This proves the raw upper bound without a convergence modulus.

For the lower bound, suppose $\Xi_f$ were a general algorithm. On $(I,0)$ it queries finitely many vector coordinates. Choose $j$ not among them. The inputs $(I,0)$ and $(I,e_j)$ agree on every queried evaluation, including all matrix entries and any supplied constants $f(m)$, but their solutions are $0$ and $e_j$. Both belong to $\Omega_f$, since $D_{f,m}(I)=0$. This contradicts finite-information dependence.
\end{proof}

\subsection{The computable adaptation: remove the exact threshold}
The fact that \eqref{eq:lin-raw} is a finite-query type-G algorithm does \textit{not} make it a Type-2 computable map of Cauchy names. The exact branch test cannot simply be imported into a Weihrauch preprocessor. The following modification is the implementation step, not a quotation of the source estimator.

\begin{lemma}[Source-proof adaptation supplied here: regularized finite sections]\label{lem:lin-regularized}
Set $\varepsilon_m=1/m$ and
\begin{equation}\label{eq:lin-reg}
 \widetilde\Gamma_m(A,b)=(C_m+\varepsilon_m I_m)^{-1}d_m.
\end{equation}
These maps are uniformly Type-2 computable from the finitely many indicated coordinate names and $m$, and converge in norm to $A^{-1}b$ on $\Omega_f$.
\end{lemma}
\begin{proof}
For every finite input matrix, $C_m+\varepsilon_m I_m\geq\varepsilon_m I_m$. Its inverse is computable by finite-dimensional arithmetic and division by a nonzero determinant; for example the positive determinant is bounded below by $\varepsilon_m^m$. Complex conjugation and all matrix operations involved are computable on Cauchy names. No test of equality to a threshold is performed. Uniformly compute a finitely supported rational complex vector $w_m$ with
\[
 \norm{w_m-\widetilde\Gamma_m(A,b)}\leq2^{-m}.
\]
For example, approximate each real and imaginary coordinate to error $2^{-m}/(2m)$. The computations use arbitrarily accurate names of only finitely many evaluation values, as permitted at name level.

For large $m$ the same coercivity bound as above gives
\[
 \norm{(C_m+\varepsilon_m I_m)^{-1}}\leq4\alpha^{-2}.
\]
Subtracting the equations for $\widetilde\Gamma_m$ and $y_m$, we obtain
\[
 \norm{\widetilde\Gamma_m-y_m} \leq4\alpha^{-2}\left(\eta_m\norm{b} +(\eta_m^2+\varepsilon_m)\norm{y_m}\right)\longrightarrow0.
\]
Since $y_m\to x$, both $\widetilde\Gamma_m$ and $w_m$ tend to $x=A^{-1}b$.
\end{proof}

\begin{lemma}[Norm-name compiler supplied here]\label{lem:hilbert-compiler}
From a uniformly computable convergent sequence $(w_m)$ of finitely supported rational complex vectors, one can compute a single array whose $\Lim$ is a $\delta_H$-name of its limit, without a convergence rate.
\end{lemma}
\begin{proof}
Reindex the sequence starting at $0$. For $s,k\in\N$, let
\[
 n(k,s)=\min\{n\leq s:\forall t\in\{n,\ldots,s\}\; \norm{w_n-w_t}\leq2^{-k-2}\}.
\]
The finite tests are decidable by squaring: squared norms of rational finite vectors are rational. As in \Cref{lem:real-limit}, $n(k,s)$ stabilizes to the least $n_k$ satisfying the inequalities for every $t\geq n_k$. Consequently the array
\[
 K(p)(\langle s,k\rangle)=\operatorname{code}(w_{n(k,s)})
\]
has eventually constant columns, and their limiting vectors satisfy $\norm{w_{n_k}-x}\leq2^{-k-2}$. Here $p$ is the input name from which all $w_m$ are computed. This is one computable functional $K$, not a nonuniform family indexed by $k$ or $s$.
\end{proof}

\subsection{Exact degree, including a separate hardness construction}
\begin{theorem}[Degree identification derived here from the source scheme]\label{thm:lin-degree}
With the representations specified above,
\[
 \Solve_f\eW\Lim,\qquad
 \pure(P_f)=\rk(\Solve_f)=1.
\]
\end{theorem}
\begin{proof}
\Cref{lem:lin-regularized} and \Cref{lem:hilbert-compiler} give a computable $K$ with
\[
 \Lim\circ K\preceq\Names_{\Solve_f}.
\]
This proves the pure one-limit upper bound, including its domain condition.

For the converse, reduce $\widehat{\LPO}$ to $\Solve_f$. Given a binary array $p(i,t)$, let $\beta_i=1$ exactly when $\exists t\;p(i,t)=1$. Fix a computable bijection $(i,t)\mapsto\langle i,t\rangle_+$ from $\N^2$ to $\N_{\geq1}$ and define $A=I$ and
\[
 b_{\langle i,t\rangle_+}=
 \begin{cases}
 2^{-(i+1)},&p(i,t)=1\text{ and }p(i,u)=0\text{ for every }u<t,\\
 0,&\text{otherwise}.
 \end{cases}
\]
Each coordinate is computed from finitely many input bits. At most one coordinate in each row $i$ is nonzero, whence
\begin{equation}\label{eq:lin-encode}
 \norm{b}^2=\sum_{i=0}^{\infty}4^{-(i+1)}\beta_i\leq\frac13.
\end{equation}
Thus $(I,b)$ is a valid instance and its product coordinate name is computable uniformly in $p$. The location of a first $1$ need not be known in advance.

The returned norm-Cauchy name of $I^{-1}b=b$ computes the real $z=\norm{b}^2$. On the promised set of reals in \eqref{eq:lin-encode}, the first base-$4$ digit is computable: digit $0$ implies $z\leq1/12$, while digit $1$ implies $z\geq1/4$. Approximation eventually distinguishes these alternatives using the separating threshold $1/6$. Output this digit, replace $z$ by $4z-\beta_0$, and iterate. The same promise holds for every remainder. This computes the entire sequence $(\beta_i)$ using only the output name, so actually
\[
 \widehat{\LPO}\sW\Solve_f.
\]
\Cref{lem:plpo} now gives $\Lim\lW\Solve_f$. Noncomputability of $\Lim$ proves the rank lower bound, and \Cref{thm:bridge} gives the equality of pure height and rank.
\end{proof}

We record shortly on how P1-P5 are implemented here explicitly:\\
P1 fixes the product coordinate names and norm-Cauchy output; P2 prohibits treating complex coordinates as comparison atoms. P3 is verified by the regularized finite-dimensional computation, not by the discontinuous threshold in the raw source estimator. On the input domain with its subspace product topology, each $\widetilde\Gamma_m$ is a continuous fixed-query map; thus this also gives $\operatorname{SCI}_{\mathrm G,\mathrm{Cont}}(P_f)=1$, using the raw lower bound. The single functional $K$ and the convergence proof supply the stronger, effective name-level uniformity. P4 is ordinary computable composition and data handling; P5 is \eqref{eq:nf}. The source proof yields the approximation mechanism; the implementation is adapted here, and the independent coding reduction yields exact degree completeness.

\section{Literature example 2: infinitely many ones}\label{sec:inf}

The second example is from \cite[Theorem 7.1, the $\Xi_2$ assertion; proof in Section 14, Step I]{BAHNS}, where the authors prove raw SCI~$2$ for deciding whether a binary sequence contains infinitely many nonzero entries. This is a non-spectral decision problem in its own right, even though it is also useful for spectral lower bounds. 

\subsection{Problem and checked raw classification}
Let
\[
 \Omega=\Cantor,\qquad M=2\text{ with the discrete metric},\qquad
 \Lambda=\{\lambda_j:j\in\N\},\quad\lambda_j(a)=a(j),
\]
and set
\[
 \INF(a)=
 \begin{cases}
 1,&\forall N\;\exists j\geq N\quad a(j)=1,\\
 0,&\text{otherwise}.
 \end{cases}
\]
All bits together determine $a$, so consistency is immediate. The evaluation image with its subspace product Cauchy representation is computably isomorphic to Cantor space: a Cauchy name of a promised $0/1$ value allows one to distinguish it from the other value by approximating to error less than $1/4$, and the reverse conversion is computable.

\begin{proposition}[Source reconstruction: {\cite[Theorem 7.1, $\Xi_2$]{BAHNS}}]\label{prop:inf-raw}
$\SCIG(P_{\INF})=2$.
\end{proposition}
\begin{proof}
Define
\begin{equation}\label{eq:inf-tower}
 \Gamma_{m,n}(a)=\mathbf1_{\{\sum_{j=0}^{n}a(j)>m\}}.
\end{equation}
Its fixed query set is $\{\lambda_0,\ldots,\lambda_n\}$. For fixed $m$, the inner sequence is eventually constant and
\[
 \Gamma_m(a):=\lim_n\Gamma_{m,n}(a) =\mathbf1_{\{a\text{ has more than }m\text{ ones}\}}.
\]
If $a$ has finitely many ones, then $\Gamma_m(a)=0$ for all sufficiently large $m$; if it has infinitely many, $\Gamma_m(a)=1$ for all $m$. Hence
\[
 \INF(a)=\lim_m\lim_n\Gamma_{m,n}(a),
\]
which proves the upper bound. It is important that the intermediate maps $\Gamma_m$ are not required to be finite-query general algorithms.

Suppose, for a contradiction, that $(\Delta_n)_n$ is a height-one raw tower computing $\INF$. Inductively choose strictly increasing $n_k$ and distinct, increasing positions $i_k$. At stage $k$, let $a^{(k)}$ have ones exactly at the previously chosen positions $i_0,\ldots,i_{k-1}$. It has finitely many ones, so choose $n_k>n_{k-1}$ with $\Delta_{n_k}(a^{(k)})=0$. Let $S_k$ be the finite set of positions queried by this algorithm on $a^{(k)}$. Choose $i_k$ larger than all previously chosen positions and all positions in $S_0\cup\cdots\cup S_k$.

Let $a$ have ones exactly at the positions $i_k$, for all $k$. For each $k$, the final $a$ agrees with $a^{(k)}$ on $S_k$: every newly added one lies beyond every index in $S_k$. Finite-information dependence therefore gives
\[
 \Delta_{n_k}(a)=\Delta_{n_k}(a^{(k)})=0.
\]
But $a$ has infinitely many ones, so convergence in the discrete target to $\INF(a)=1$ is impossible. This rules out height~$1$ and, by padding, height~$0$.
\end{proof}

The source proof does not require an effective construction of the adversarial sequence. The point is that it works against every raw general tower, including noncomputable ones. In contrast, the upper family \eqref{eq:inf-tower} is uniformly computable from the input bits and both indices.

\subsection{Exact Weihrauch degree: the derivative of LPO}
An equivalent concrete presentation of $\LPO'$ has input a binary array $p_s(j)$, promised eventually constant in $s$ for each $j$, and output
\[
 \LPO'(p)=\LPO(u),\qquad u(j)=\lim_s p_s(j).
\]
This is equivalent to the literal jumped-representation definition: intermediate natural-valued approximations can be replaced by their indicators of equality to $1$, since their final values are promised to be in $2$.

\begin{theorem}[Exact degree and rank distinctions derived here]\label{thm:inf-degree}
\[
 \INF\esW\LPO',\qquad
 \pure(P_{\INF})=\rk(\INF)=2.
\]
Moreover,
\begin{equation}\label{eq:inf-strict}
 \INF\LW\Lim^{(2)},\qquad
 \INF\nW\Lim,\qquad \Lim\nW\INF.
\end{equation}
\end{theorem}

\begin{proof}[Two reductions supplied here, not extracted from {\cite{BAHNS}}]
Write $\FIN(a)=1-\INF(a)$. Complementing a discrete output is computable, so $\FIN\esW\INF$.

For $\FIN\sW\LPO'$, given $a$, compute
\[
 p_s(n)=\mathbf1_{\{\forall k\in\{n,\ldots,s\}\quad a(k)=0\}},
\]
where an empty range satisfies the condition. For each $n$, this sequence stabilizes to $1$ exactly when the entire tail from $n$ is zero. Thus
\[
 \LPO(\Lim(p))=1 \ \Longleftrightarrow\ 
 \exists n\;\forall k\geq n\quad a(k)=0 \ \Longleftrightarrow\ \FIN(a)=1.
\]
The preprocessor makes only finitely many bit queries for each output coordinate, and the postprocessor is the identity.

For $\LPO'\sW\FIN$, given $p_s(j)$, put
\[
 w_s=\min\bigl(\{j\leq s:p_s(j)=1\}\cup\{s+1\}\bigr)
\]
and compute a binary sequence $a$ by
\[
 a(0)=0,\qquad a(s+1)=\mathbf1_{\{w_{s+1}\neq w_s\}}.
\]
Let $u(j)=\lim_s p_s(j)$. If $u$ contains a $1$, let $j_*$ be its least position. Eventually every position below $j_*$ is $0$ and position $j_*$ is $1$, so $w_s=j_*$ permanently and $a$ has only finitely many ones. If $u$ is identically zero, then, for every fixed $J$, eventually all positions $j\leq J$ are zero, so $w_s>J$. Thus $w_s\to\infty$ and it changes infinitely often; $a$ has infinitely many ones. We have proved
\[
 \FIN(a)=\LPO'(p).
\]
Again this is a strong reduction. These two constructions, rather than the integer $2$ in the SCI theorem, identify the exact degree.
\end{proof}

\begin{proof}[Pure upper bound, rank lower bound, and strictness]
The source tower gives an explicit pure two-limit preprocessor. After decoding each promised evaluation value into a bit, set
\[
 K(a)(\pair{n,\pair{m,k}})=\Gamma_{m,n}(a)\qquad(n,m,k\in\N).
\]
The first application of $\Lim$ produces the array $(m,k)\mapsto\Gamma_m(a)$; the second produces the constant sequence with value $\INF(a)$. All inner and outer domain requirements hold by the preceding convergence proof. This constant sequence is a valid discrete output name. Hence $\INF\lW\Lim^{(2)}$.

Suppose $\INF\lW\Lim$. The normal form \eqref{eq:nf}, applied to the identity representation of Cantor space, yields continuous binary maps $g_s:\Cantor\to2$ with $g_s(a)\to\INF(a)$ for every $a$. Indeed, take the zeroth output-name coordinate of the approximating array and map any intermediate natural number other than $1$ to $0$; the eventual value is correct.

For $N\in\N$ and $i\in2$, the sets
\[
 F_N^i=\bigcap_{s\geq N}\{a:g_s(a)=i\}
\]
are closed and cover Cantor space as $(N,i)$ varies. By the Baire category theorem, one has nonempty interior. On that interior $\INF$ is constantly $i$. This is impossible: every nonempty cylinder contains a sequence with finitely many ones and one with infinitely many ones. Therefore $\INF\nW\Lim$, and the pure rank is exactly~$2$.

Finally, $\Lim\nW\INF$. A computable $\Lim$-input can have a noncomputable output, for example the halting-set characteristic sequence obtained as a limit of computable finite-stage indicators. A hypothetical reduction to $\INF$ would transform that computable input into a computable $\INF$-instance. Its answer is one bit, with a computable canonical name. Using that allowed oracle answer, the computable postprocessor would give a computable $\Lim$-output, a contradiction. Since $\Lim\lW\Lim^{(2)}$, also $\Lim^{(2)}\nW\INF$. Together with the upper bound this gives all of \eqref{eq:inf-strict}.
\end{proof}

We record again the explicit implementation of P1-P5 in this example:\\ 
P1 is the bit-evaluation product representation, computably equivalent to Cantor space, and discrete output representation. P2 causes no loss because the values are promised to be $0$ or $1$. P3 consists only of finite sums of bits, comparisons of integers, and computable indeces. The displayed $K$ supplies global uniformity, and the tower proof verifies the complete iterated-limit domain condition. P4 and P5 are the same Type-2 closure and normal form as before. The finite bit maps are also continuous in the product topology, so the regularity-restricted raw height $\operatorname{SCI}_{\mathrm G,\mathrm{Cont}}$ is $2$. The finite-injury proof verifies the source raw lower bound; the Baire-category argument is a separate represented rank lower bound; the reductions involving $\FIN$ establish the exact Weihrauch degree.

\begin{remark}[Why this example is especially diagnostic]
The equality $\pure(P_{\INF})=2$ is a real successful bridge application. Nonetheless the exact degree is not $\Lim^{(2)}$. Indeed it is even incomparable with the lower-rank degree $\Lim$. A numerical rank is monotone under reducibility but does not linearly order the degrees that receive those numerical values.
\end{remark}

\clearpage
\part{IVT in the traced finite-query transport hierarchy}\label{part:transport}

\section{What  extension is needed here}\label{sec:transport-framework}

\subsection{Source scope: transport is not already a multivalued SCI definition}
The two requested papers address a different axis from the pure-height bridge. In \cite[Definition 2.1]{SorgWitness}, the target is still a \textit{single-valued} map $\Xi:\Omega\to M$. Its Definitions 4.8 and Theorem 4.10 concern finite-query transports and preservation of raw height. Likewise, \cite[Definitions 2.1 and 5.1]{SorgTransport} uses single-valued targets, including in its typed formulation. The word ``typed'' refers to evaluation value spaces; it does not mean ``multivalued''. Its Definition 5.10 and Theorem 5.1 characterize TTE transport by a uniform finite interface trace and then forget that trace to obtain strong Weihrauch reducibility. Corollary 5.1 explicitly refutes the converse.

Consequently these works are very suitable for the \textit{transport question}, but neither can simply be cited as a multivalued version of the foundational SCI definition. We make the necessary extension below, prove the properties used, and mark this extension as supplied here. The ordinary IVT target remains \textit{one arbitrary zero}, not the entire zero set as a hyperspace-valued output.

\subsection{The multivalued raw and transport definitions}
\begin{definition}[Multivalued extension supplied here]\label{def:mv}
A multivalued typed SCI problem is $P=(F_P,\Omega_P,(M_P,d_P),\Lambda_P)$, where $F_P:\Omega_P\rightrightarrows M_P$ has nonempty values. We restrict to nonempty, countable, effectively indexed interfaces. Each $\lambda\in\Lambda_P$ has a specified value space $V_\lambda$. We impose the extensional consistency condition
\[
 \bigl[\forall\lambda\in\Lambda_P\quad\lambda(x)=\lambda(y)\bigr] \ \Longrightarrow\ F_P(x)=F_P(y).
\]
A raw tower consists of single-valued deepest general algorithms, as in \Cref{def:raw}, whose successive metric limits exist and whose final value belongs to $F_P(x)$ for every input. Its least height is denoted $\SCImv(P)$, with value $\infty$ if no finite tower exists. At height zero a general algorithm must be a selector of $F_P$.

For the represented statements below, we additionally require that the full evaluation table $\Ev_P$ separates input points. Give its image $I_{\Lambda_P}$ the subspace product representation of the evaluation-value spaces, and give $\Omega_P$ the representation transported back along the bijection $\Ev_P:\Omega_P\to I_{\Lambda_P}$. Thus the represented instance and its full information table are computably isomorphic by definition; no inverse between unrelated representations is being assumed. All examples below verify this convention explicitly. The factor relation is
\[
 \widehat F_P(\Ev_P(x))=F_P(x).
\]
For this represented problem define the pure multivalued height by
\[
 \puremv(P)=\min\{k:\exists K\in\Comp\quad \Lim^{(k)}K\preceq\Names_{\widehat F_P}\}.
\]
An empty set of admissible pure heights again has minimum $\infty$. The represented normal form in \cite[Theorem 4.27]{SorgBridge} applies to relations, so exactly the same minimum argument gives
\begin{equation}\label{eq:mv-bridge}
 \puremv(P)=\rk(\widehat F_P).
\end{equation}
\end{definition}

\begin{definition}[Multivalued finite-query transport; extension of {\cite[Definition 4.8]{SorgWitness}} and {\cite[Definitions 5.4, 5.9-5.10]{SorgTransport}}]\label{def:mv-transport}
Write $S\leq^{\mathrm{mv}}_{\mathrm{raw},\mathrm{fq}}P$ if there are a point map $E:\Omega_S\to\Omega_P$, a \textit{total continuous metric-space map} $D:M_P\to M_S$, and, for each indexed target evaluation $\lambda_e$, a finite list of source evaluations $\gamma_{i(e,1)},\ldots,\gamma_{i(e,m_e)}$ and a map $\vartheta_e$ on their transcript image, such that
\begin{align}
 \varnothing\neq D[F_P(E(x))]&\subseteq F_S(x),\label{eq:mv-sound}\\
 \lambda_e(E(x))&=\vartheta_e\bigl(\gamma_{i(e,1)}(x),\ldots, \gamma_{i(e,m_e)}(x)\bigr).\label{eq:trace}
\end{align}
Here $m_e=0$ is allowed for constants; a fixed existing source evaluation can be queried and ignored to restore the positive-query convention. The finite index list depends on $e$, not on $x$. Further, the decoder receives \textit{only the target answer}, not the source input. Inclusion in \eqref{eq:mv-sound}, rather than equality of solution sets, is the soundness condition needed for arbitrary permitted oracle outputs.

With the evaluation-induced instance representations of \Cref{def:mv}, and fixed representations of the evaluation values and outputs, write $S\leq^{\mathrm{mv}}_{\mathrm{TTE},\mathrm{fq}}P$ if in addition $E,D$ have computable realizers, and a single computable procedure, given $e$, produces the finite index list and an index for a partial computable reconstruction $\widehat\vartheta_e$ whose domain contains the whole transcript image and which satisfies \eqref{eq:trace}.

We retain metric continuity of $D$ explicitly. In our constructions $M_P=M_S=[0,1]$ with ordinary Cauchy names and $D=\id$, so this compatibility is automatic. It would not be justified for arbitrary unrelated output representations merely by writing ``TTE''.
\end{definition}

\begin{theorem}[Pullback and trace forgetting; proofs extended here from {\cite[Lemma 4.9, Theorem 4.10]{SorgWitness}} and {\cite[Theorem 5.1]{SorgTransport}}]\label{thm:mv-transport}
The relations in \Cref{def:mv-transport} are preorders, and
\begin{align}
 S\leq^{\mathrm{mv}}_{\mathrm{raw},\mathrm{fq}}P
 &\ \Longrightarrow\ \SCImv(S)\leq\SCImv(P),\label{eq:raw-pullback}\\
 S\leq^{\mathrm{mv}}_{\mathrm{TTE},\mathrm{fq}}P
 &\ \Longrightarrow\ \widehat F_S\sW\widehat F_P
 \ \Longrightarrow\ \widehat F_S\lW\widehat F_P.\label{eq:forget}
\end{align}
Accordingly the second hypothesis also implies $\rk(\widehat F_S)\leq\rk(\widehat F_P)$ and $\puremv(S)\leq\puremv(P)$.
\end{theorem}
\begin{proof}[Reconstructed pullback, with the multivalued change explicit]
For each $\lambda\in\Lambda_P$, fix a representative index
\[
c_P(\lambda)\in\mathbb N, \qquad \lambda_{c_P(\lambda)}=\lambda.
\]
For example, take the least such index in the fixed enumeration. No computability of $c_P$ is required for this raw pullback argument. Fix also $\gamma_\circ\in\Lambda_S$.

Given a deepest general algorithm $(\Gamma,Q_\Gamma)$ for $P$, define
\[
\widetilde\Gamma(x)=D(\Gamma(E(x)))
\]
and
\[
Q^0_{\widetilde\Gamma}(x)=
\bigcup_{\lambda\in Q_\Gamma(E(x))}
\left\{ \gamma_{i(c_P(\lambda),j)}: 1\leq j\leq m_{c_P(\lambda)} \right\}.
\]
Set
\[
Q_{\widetilde\Gamma}(x)=
\begin{cases}
	Q^0_{\widetilde\Gamma}(x),
	&Q^0_{\widetilde\Gamma}(x)\neq\varnothing,\\
	\{\gamma_\circ\},
	&Q^0_{\widetilde\Gamma}(x)=\varnothing.
\end{cases}
\]

The union is finite because it is indexed by the finite set $Q_\Gamma(E(x))$, with one finite trace for each member.

Suppose that $y$ agrees with $x$ on $Q_{\widetilde\Gamma}(x)$. For every $\lambda\in Q_\Gamma(E(x))$, its chosen reconstruction gives
\[
\lambda(E(y))=\lambda(E(x)).
\]
The general-algorithm axioms for $\Gamma$ therefore imply
\[
\Gamma(E(y))=\Gamma(E(x)), \qquad
Q_\Gamma(E(y))=Q_\Gamma(E(x)).
\]
Consequently
\[
\widetilde\Gamma(y)=\widetilde\Gamma(x), \qquad
Q^0_{\widetilde\Gamma}(y)=Q^0_{\widetilde\Gamma}(x).
\]
Padding is the same at both inputs, so
\[
Q_{\widetilde\Gamma}(y)=Q_{\widetilde\Gamma}(x).
\]
Thus $(\widetilde\Gamma,Q_{\widetilde\Gamma})$ is a general algorithm.

Apply this construction at every deepest tower index. Continuity of $D$ carries each existing metric limit to its image under $D$, successively from the innermost to the outermost limit. If the original final value is $z\in F_P(E(x))$, the transported final value is $D(z)\in F_S(x)$ by \eqref{eq:mv-sound}. At height zero the same argument applies without a limit. This proves \eqref{eq:raw-pullback}.
\end{proof}
\begin{proof}[Trace forgetting, proved at the realizer level]
Use the computable evaluation-table isomorphisms of \Cref{def:mv} to view $E$ as a map between information spaces. Let $\Phi_E,\Phi_D$ be computable realizers of that map and of the output decoder, respectively. For \textit{every} realizer $G$ of $\widehat F_P$ and every source input name $p$, the name $G(\Phi_E(p))$ denotes some $z\in F_P(E(x))$, where $p$ names $x$. Therefore
\[
 \Phi_D\circ G\circ\Phi_E(p)
\]
names $D(z)\in F_S(x)$. This is a strong reduction: $\Phi_D$ has no access to $p$. The trace itself is not needed for this implication once the point realizers have been established; it is additional structure that the Weihrauch reduction forgets.

For completeness, identities have one-query identity traces. To compose two transports, compose their encodings and their decoders in the opposite order. The two solution inclusions compose. Substitute each finite trace of the second encoding into the traces of the first and concatenate the resulting finite lists. The reconstruction is the corresponding composition on transcript images. In the TTE case all lists and machine indices are computed uniformly. Metric continuity and computability of the decoders are preserved. Hence both relations are preorders. Ordinary degree and rank monotonicity now follow by transitivity and \eqref{eq:mv-bridge}.
\end{proof}

\begin{proposition}[The converse fails even for single-valued identities: {\cite[Corollary 5.1]{SorgTransport}}; proof reconstructed]\label{prop:trace-not-converse}
There are strongly Weihrauch-equivalent represented targets whose SCI problems are not finite-query reducible in one direction, even in the raw modality.
\end{proposition}
\begin{proof}
Let $\tau(b)=\sum_{n\geq0}2b(n)3^{-(n+1)}$ on $\Cantor$. Use the same identity target $\id_{\Cantor}$, standard represented input and output, and compare the interfaces
\[
 \Lambda_S=\{b\mapsto b(n):n\in\N\},\qquad
 \Lambda_P=\Lambda_S\cup\{\tau\}.
\]
The represented targets are identical. A hypothetical transport $S\leq_{\mathrm{raw},\mathrm{fq}}P$ has $D(E(b))=b$, so $E$ is injective. The target query $\tau$ must then factor through finitely many source bits. Two distinct sequences agreeing on these bits give $\tau(E(b))=\tau(E(c))$. Injectivity of $\tau$ and then of $E$ contradicts $b\neq c$.

One can also see the raw-height obstruction directly: truncating the bit sequence proves $\SCIG(S)=1$, and changing an unqueried bit rules out height zero; the single exact query $\tau$ makes $\SCIG(P)=0$ by unrestricted inverse postprocessing. The representations induced by the two evaluation tables are computably equivalent, but their finite exact-query interfaces are not interchangeable. Thus a Weihrauch equivalence must not be promoted to a traced equivalence without proving the traces.
\end{proof}

\subsection{Ranks, transport degrees, and the correct forgetful maps}
For a family $\mathcal U$ define
\[
 \mathcal O^{\mathrm{mv}}_k(\mathcal U)=\{P\in\mathcal U:\SCImv(P)=k\},\qquad
 \mathcal D^{\mu,\mathrm{mv}}_k(\mathcal U)= \mathcal O^{\mathrm{mv}}_k(\mathcal U)/{\equiv^{\mathrm{mv}}_{\mu,\mathrm{fq}}}.
\]
This is the multivalued adaptation of \cite[Definition 7.1]{SorgTransport}. On the TTE quotient, \Cref{thm:mv-transport} gives well-defined order-preserving maps
\begin{equation}\label{eq:degree-chain}
 [P]_{\mathrm{TTE},\mathrm{fq}}
 \longmapsto[\widehat F_P]_{\mathrm{sW}}
 \longmapsto[\widehat F_P]_{\mathrm W}
 \longmapsto\rk(\widehat F_P).
\end{equation}
Separately, forgetting effectiveness sends a TTE class to its raw finite-query class, and raw height is invariant on raw classes. There is \textit{no asserted inverse} to any map in \eqref{eq:degree-chain}. In particular, transport degrees are classes of problems \textit{with interfaces}, whereas Weihrauch degrees forget those interfaces. The pure bridge equates two notions of rank; the transport comparison relates two reducibilities.

\section{IVT and interval choice: the represented ingredients}\label{sec:ivt-basics}

\subsection{The exact target is a choice problem}
Let
\[
 \mathcal D=\{h\in C([0,1],\R):h(0)<0<h(1)\},\qquad
 \IVT(h)=\{x\in[0,1]:h(x)=0\}.
\]
The domain representation $\delta_C$ gives rational polygonal approximations $P_n$ with $\norm{P_n-h}_\infty\leq2^{-n}$, and the output representation is $\rho$. Strict endpoint inequalities are a promise, not additional noncomputable advice. The opposite orientation can be reduced to this one by multiplying by $-1$ after semideciding the endpoint signs.

The target is \textit{any one zero}. It is not the full zero set in Hausdorff distance and not the leftmost zero. Those are different problems. A set such as $h^{-1}(0)$ can of course be a single point in a hyperspace, but changing the output to that set changes the computational task.

\subsection{A self-contained version of the classical IVT classification}
Let $\CC$ be closed interval choice: an input is a nonempty closed interval $J=[a,b]\subseteq[0,1]$, possibly a singleton, given by \textit{negative information}; an output is any $x\in J$, in representation $\rho$. Negative information is an enumeration of rational relative open intervals whose union is $[0,1]\setminus J$, with a padding symbol that enumerates nothing.

\begin{lemma}[Interval-name normal form: elementary proof supplied here]\label{lem:interval-name}
From a negative name of $J=[a,b]$, one can uniformly compute rational bounds
\[
 0\leq a_s\leq a\leq b\leq b_s\leq1,\qquad
 a_s\uparrow a,\quad b_s\downarrow b.
\]
Conversely, such bounds uniformly produce a negative name of $J$.
\end{lemma}
\begin{proof}
For rational $r\in[0,1]$, the statement $[0,r]\cap J=\varnothing$ is semidecidable by searching for a finite subcover of $[0,r]$ from the enumerated complement. Finite coverage by rational open intervals is decidable. Dovetail these tests, retaining the largest certified lower bound together with $0$. Do the analogous construction for $[r,1]$ and retain the smallest certified upper bound together with $1$. If $r<a$, the first test eventually succeeds, and if $r>b$, the second does. The resulting bounds converge to the endpoints, including the cases $a=0$ or $b=1$. Conversely, enumerate $[0,a_s)$ and $(b_s,1]$ as relative open sets. Their unions are exactly the two components of the complement.
\end{proof}

\begin{theorem}[Classical degree: {\cite[Theorem 6.2]{BG}}; proof provenance below]\label{thm:ivt-cc}
\[
 \IVT\esW\CC,\qquad \one\LW\CC\LW\Lim.
\]
The equivalence is the result of Brattka and Gherardi \cite[Theorem 6.2]{BG}; their notation for interval choice is $C_I$.
\end{theorem}

\begin{proof}[Adaptation of the signed-bracket argument in {\cite[Theorem 6.2]{BG}}]
Begin with the rational bracket $[a_0,b_0]=[0,1]$, whose endpoint values have the promised strict signs. Dovetail the semidecision procedures for $h(r)<0$ and $h(r)>0$ at all rational $r\in[0,1]$. Process their successful certificates fairly. Whenever a certificate concerns a rational $r$ strictly inside the current bracket, replace the left endpoint by $r$ if $h(r)<0$, or the right endpoint by $r$ if $h(r)>0$. Otherwise keep the bracket unchanged. Use finite bounded simulations per stage, so every stage of this nested-bracket construction is computable.

Let $J=\bigcap_s[a_s,b_s]=[a,b]$. The endpoint signs are invariant. If $a=b$, continuity and the two endpoint sequences imply $h(a)=0$. If $a<b$, every rational $r\in(a,b)$ belongs to the interior of every bracket. Were $h(r)\neq0$, its certificate would eventually be processed, replacing an endpoint by $r$ and contradicting $r\in(a,b)$. Thus $h$ vanishes on all rationals in $(a,b)$ and, by continuity, on all of $[a,b]$. Consequently
\[
 \varnothing\neq J\subseteq h^{-1}(0).
\]
The nested rational brackets compute a negative name of $J$. Every choice from $J$ is already a zero of $h$, so the postprocessor is the identity. This reduction uses only rational point values of $h$, not a computable modulus.
\end{proof}

\begin{proof}[Hinge-series reverse reduction supplied here]
Obtain the bounds of \Cref{lem:interval-name}. Move the interval into the interior by setting
\[
 A_s=\tfrac13+\tfrac13 a_s,\qquad B_s=\tfrac13+\tfrac13 b_s,
\]
with limits $A=(1+a)/3$ and $B=(1+b)/3$. Define
\begin{equation}\label{eq:interval-to-ivt}
 h(x)=\sum_{s=0}^{\infty}2^{-(s+1)} \bigl(\pos{x-B_s}-\pos{A_s-x}\bigr),
 \qquad \pos{u}=\max\{u,0\}.
\end{equation}
The summands are rational piecewise linear and bounded in absolute value by $2^{-(s+1)}$, so partial sums give a uniform polygonal name with an effective geometric error bound. If $x\in[A,B]$, all terms vanish. If $x<A$, no positive term occurs and at least one negative term is strictly negative; if $x>B$, the symmetric assertion holds. Therefore
\[
 h^{-1}(0)=[A,B],\qquad h(0)<0<h(1).
\]
Given any output $x$ of IVT, return $3x-1\in[a,b]$. This is again a strong reduction.
\end{proof}

\begin{proof}[Supplementary proofs supplied here: noncomputability and strictness]
An enumeration containing only padding symbols is a negative name of $J=[0,1]$. If interval choice had a computable realizer, one of its output rational approximations at precision, say, $2^{-3}$ would be produced after finitely many input queries. Preserve those padding answers and extend the name first to a name of $\{0\}$ and, separately, to a name of $\{1\}$. The same rational would have to approximate both $0$ and $1$ within $1/8$, which is impossible. Thus $\CC$ is noncomputable.

To obtain $\CC\sW\Lim$, compute $a_s\uparrow a$ by \Cref{lem:interval-name}, apply \Cref{lem:real-limit}, and return the left endpoint $a\in J$.

Every computable interval-choice instance has at least one computable solution: a nondegenerate interval contains a rational, while for a singleton the computable bounds $a_s,b_s$ converge to the same point, and searching for $b_s-a_s<2^{-k}$ computes that point. This is a nonuniform existence statement, not a choice algorithm. A computable $\Lim$-instance can have a unique noncomputable output. If $\Lim\lW\CC$, a computable preprocessing of that instance would have a computable interval-choice solution, and the postprocessor applied to that solution would make the $\Lim$-output computable. This contradiction proves $\Lim\nW\CC$.
\end{proof}

\subsection{Why the jumped IVT degree has pure rank two}

\begin{proposition}[Jump consequences derived here; standard ingredients {\cite{BGM}}]\label{prop:ivt-jump}
\[
 \IVT'\esW\CC',\qquad
 \one\LW\CC\LW\CC',\qquad
 \rk(\CC)=1,\quad\rk(\CC')=2.
\]
\end{proposition}
\begin{proof}[Derived consequences, with standard ingredients identified]
The jumped equivalence follows from \Cref{thm:ivt-cc} and \Cref{lem:lift}. Constant sequences of names show $\CC\lW\CC'$. The strong reduction $\CC\sW\Lim$ also lifts, giving
\[
 \CC'\sW\Lim'\equiv_{\mathrm{sW}} \Lim\circ\Lim=\Lim^{(2)}.
\]
We prove $\CC'\nW\Lim$ explicitly; this also proves $\CC'\nW\CC$.

Let $C_2$ be negative-information choice from a nonempty subset of $2$. There is a strong reduction $C_2\sW\CC$: start with $J=[0,1]$; if $1$ is excluded, shrink to $[0,1/3]$; if $0$ is excluded, shrink to $[2/3,1]$. Both distinct exclusions cannot occur on a valid input. For any chosen $x$, race the semidecisions $x<2/3$ and $x>1/3$. In the first case return $0$, in the second $1$. At least one test succeeds, and any possible returned bit is allowed by the original set. The postprocessor does not require the original name.

Define the binary cluster-point problem
\[
 \BWT(a)=\{i\in2:\forall N\;\exists t\geq N\quad a(t)=i\}.
\]
We have $\BWT\sW C_2'$. Indeed, negative names of subsets of $2$ may use $0$ as padding and $i+1$ to exclude $i$. Given $a$, compute
\[
 q_s(\pair{i,N})=
 \begin{cases}
 i+1,&\forall t\in\{N,\ldots,s\}\quad a(t)\neq i,\\
 0,&\text{otherwise}.
 \end{cases}
\]
Each coordinate stabilizes. Its limit enumerates an exclusion of $i$ exactly when there exists a tail containing no $i$, that is, when $i$ is not a cluster value. The named subset is nonempty. Lifting $C_2\sW\CC$ therefore yields
\[
 \BWT\sW C_2'\sW\CC'.
\]

The fact $\BWT\nW\Lim$ is \cite[Proposition 12.5]{BGM}. Here is a direct category proof for the present use, rather than an extraction of that source proof. Otherwise the pure normal form would give continuous binary maps $g_s:\Cantor\to2$ (project the first coordinate of the approximating name to $2$, assigning a default bit to any temporary invalid symbol) converging pointwise to a selector $g$ of $\BWT$. The closed sets $F_N^i=\bigcap_{s\geq N}\{a:g_s(a)=i\}$ cover Cantor space, so one has nonempty interior by the Baire category theorem. On that interior $g=i$. But every nonempty cylinder contains a sequence eventually constantly $1-i$, on which the only permitted cluster value is $1-i$. This is a contradiction. Consequently $\CC'\nW\Lim$, and all the stated rank and strictness claims follow.
\end{proof}

\subsection{Two fixed name transformers that will acquire finite traces}\label{sec:name-transformers}
For the transport construction fix the subspaces
\[
 \mathcal H=\{h\in\mathcal D:h^{-1}(0)\subseteq[1/3,2/3]\},\qquad
 \mathcal J=\{[a,b]:1/3\leq a\leq b\leq2/3\}.
\]
The first has the restricted polygonal representation $\delta_H$; the second has the restricted negative-information representation $\delta_J$. Let $D_H=\dom(\delta_H)$ and $D_J=\dom(\delta_J)$. Both restrictions retain the usual IVT/interval-choice degree. For interval choice, a computable affine contraction moves an arbitrary interval into $[1/3,2/3]$, and an affine decoder reverses it. For IVT, the contracted hinge construction in \Cref{thm:ivt-cc} already outputs a function in $\mathcal H$, and the signed-bracket direction still applies on that restricted domain. Thus each restriction is strongly Weihrauch-equivalent to $\CC$.

\begin{lemma}[Name transformers extracted and adapted from the preceding proofs]\label{lem:alpha-beta}
There are fixed partial computable Baire-space maps
\[
 \alpha:D_H\to D_J,\qquad \beta:D_J\to D_H,\qquad
 \operatorname{sgl}:\rho^{-1}([1/3,2/3])\to D_J
\]
such that
\begin{align}
 \delta_J(\alpha(u))&\subseteq\delta_H(u)^{-1}(0),\label{eq:alpha}\\
 \delta_H(\beta(q))^{-1}(0)&=\delta_J(q),\label{eq:beta}\\
 \delta_J(\operatorname{sgl}(v))&=\{\rho(v)\}.\label{eq:sgl}
\end{align}
There are also partial computable lifted transformers $\alpha^\Delta,\beta^\Delta$ satisfying, on their promised domains,
\begin{equation}\label{eq:exact-lift}
 \Lim(\alpha^\Delta(p))=\alpha(\Lim(p)),\qquad
 \Lim(\beta^\Delta(q))=\beta(\Lim(q)).
\end{equation}
The equalities in \eqref{eq:exact-lift} are equalities of Baire names, not just of the named mathematical objects.
\end{lemma}
\begin{proof}
Fix deterministic fair scheduling in the signed-bracket construction of \Cref{thm:ivt-cc}; its output negative name defines $\alpha$. Its nonempty limiting interval consists of zeros and is therefore contained in $[1/3,2/3]$.

For $\beta$, compute the nested rational bounds of \Cref{lem:interval-name}, then replace them by $a_s\gets\max(1/3,a_s)$ and $b_s\gets\min(2/3,b_s)$. They still converge monotonically to the endpoints of $J=\delta_J(q)$. Use the \textit{uncontracted} hinge series
\begin{equation}\label{eq:hinge-transport}
 h_q(x)=\sum_{s=0}^\infty2^{-(s+1)} \bigl((x-b_s)_+-(a_s-x)_+\bigr).
\end{equation}
Its tail after $N$ terms has uniform norm at most $2^{-N}$. The sign argument already proved gives $h_q^{-1}(0)=J$. Moreover $h_q(0)\leq-1/3$ and $h_q(1)\geq1/3$, so $h_q\in\mathcal H$. Computing the rational polygonal partial sums defines $\beta$.

From a Cauchy name of $r$, enumerate rational relative open intervals whose closures stay a certified positive distance from $r$. Their union is $[0,1]\setminus\{r\}$; dovetailing the strict distance tests, with padding, defines $\operatorname{sgl}$. Finally apply the bounded-simulation construction of \Cref{lem:lift} to $\alpha$ and $\beta$.
\end{proof}

\begin{remark}[A necessary point-extensionality distinction]\label{rem:names-are-points}
The maps $\alpha$ and $\beta$ are fixed functions of \textit{names}. Different names of the same function may produce different zero intervals; different names of the same interval may produce different functions. Thus these constructions are not being asserted to be point maps on unnamed functions or unnamed intervals. In the SCI packages below the payload name itself is part of the mathematical input. Consequently $p\mapsto\alpha(p)$ is a point map on those name-enriched input spaces. This is essential for applying the point-extensional trace theorem honestly.
\end{remark}

\section{Six presentations of three IVT information packages}\label{sec:ivt-packages}
\subsection{One encoded payload and a separately queried infinite shift}
The construction is deliberately a toy interface. It keeps all evaluation value representations standard and uses precisely a countable product-Cauchy information space. One exact real evaluation encodes an entire payload name. Raw type~G may apply an arbitrary function to that exact value; TTE receives only a Cauchy name and may not perform the same semantic decoding for free. An independent shift, supplied bit by bit rather than as one exact real atom, prevents height-zero raw solvability.

\begin{lemma}[Computable real coding, with inverse on its range; construction supplied here]\label{lem:encoding}
There is a computable injection $\kappa:\B\to[0,1]$ whose inverse on its range is computable from Cauchy names.
\end{lemma}
\begin{proof}
For $p\in\B$ let
\[
 \eta(p)=0^{p(0)}1\,0^{p(1)}1\,0^{p(2)}1\cdots,\qquad
 \tau(b)=\sum_{n=0}^{\infty}2b(n)3^{-(n+1)},\qquad
 \kappa(p)=\tau(\eta(p)).
\]
The bound for the ternary tail after $N$ digits is $3^{-N}$, proving computability. For $z=\tau(b)$ the alternatives $b(0)=0$ and $b(0)=1$ place $z$ respectively in $[0,1/3]$ and $[2/3,1]$. Finite-precision approximation distinguishes them across the gap. Replace $z$ by $3z-2b(0)$ and repeat. This computes $b$ and proves injectivity. On $\kappa(\B)$ the decoded bit sequence has infinitely many delimiter ones, so parsing the successive zero blocks computes $p$. Range membership need not be decidable and is not being tested.
\end{proof}

Fix
\begin{equation}\label{eq:shift}
 t(b)=\frac14+\frac18\tau(b),\qquad
 t_n(b)=\frac14+\frac18\sum_{j<n}2b(j)3^{-(j+1)},\qquad \ell=\frac14.
\end{equation}
Then $1/4\leq t(b)\leq3/8$ and
\begin{equation}\label{eq:shift-error}
 0\leq t(b)-t_n(b)\leq\frac18\,3^{-n}.
\end{equation}
For $h\in\mathcal H$ set
\begin{equation}\label{eq:shifted-ivt}
 (T_bh)(x)=h\!\left(\max\left\{0,\min\left\{1,\frac{x-t(b)}{\ell}\right\}\right\}\right),
 \qquad x\in[0,1].
\end{equation}
It is continuous, strictly negative at $0$, strictly positive at $1$, and
\begin{equation}\label{eq:shifted-zeros}
 \IVT(T_bh)=t(b)+\ell h^{-1}(0).
\end{equation}
Indeed the outer constant pieces have nonzero values, and on the affine middle segment vanishing is equivalent to vanishing of $h$.

\subsection{The payload domains and their exact semantics}
Write $Z(h)=h^{-1}(0)$. Define function payload domains and their decoded functions by
\begin{align*}
 D^H_{\mathrm A}
 &=\{\pair{u,v}:u\in D_H,\ \rho(v)\in Z(\delta_H(u))\},
 &h_{\mathrm A}(\pair{u,v})&=\delta_H(u),\\
 D^H_{\mathrm B}&=D_H,
 &h_{\mathrm B}(u)&=\delta_H(u),\\
 D^H_{\mathrm C}
 &=\{p\in\dom(\Lim):\Lim(p)\in D_H\},
 &h_{\mathrm C}(p)&=\delta_H(\Lim(p)).
\end{align*}
Define interval payload domains and their decoded intervals by
\begin{align*}
 D^J_{\mathrm A}
 &=\{\pair{q,v}:q\in D_J,\ \rho(v)\in\delta_J(q)\},
 &J_{\mathrm A}(\pair{q,v})&=\delta_J(q),\\
 D^J_{\mathrm B}&=D_J,
 &J_{\mathrm B}(q)&=\delta_J(q),\\
 D^J_{\mathrm C}
 &=\{q\in\dom(\Lim):\Lim(q)\in D_J\},
 &J_{\mathrm C}(q)&=\delta_J(\Lim(q)).
\end{align*}
Thus A includes a supplied solution, B has an ordinary name, and C has a jumped ordinary name. Only the limiting name in C must be valid. No promise is made about its intermediate rows.

For $j\in\{\mathrm A,\mathrm B,\mathrm C\}$ define the SCI problems
\begin{align}
 P_j &: \quad \Omega^H_j=D^H_j\times\Cantor,
 &F_j(p,b)&=t(b)+\ell Z(h_j(p)),\label{eq:Pj}\\
 Q_j &: \quad \Omega^J_j=D^J_j\times\Cantor,
 &G_j(q,b)&=t(b)+\ell J_j(q).\label{eq:Qj}
\end{align}
The output metric space in every case is $([0,1],|\cdot|)$; the output representation is $\rho$. Each $P_j$ asks for an ordinary IVT zero of the function \eqref{eq:shifted-ivt}. Each $Q_j$ asks for a point in an interval, not for its endpoints or a name of the whole interval.

Every one of the six problems has the same \textit{form} of evaluation interface,
\begin{equation}\label{eq:payload-interface}
 \lambda_*(w,b)=\kappa(w),\qquad
 \lambda_n(w,b)=b(n)\quad(n\in\N).
\end{equation}
Index $*$ by $0$ and $n$ by $n+1$. All values are real numbers with ordinary Cauchy names. The full evaluation table is injective. Its image
\[
 I_D=\{(\kappa(w),b(0),b(1),\ldots):w\in D,\ b\in\Cantor\} \subseteq\R^{\N}
\]
has the subspace product-Cauchy representation. It is computably isomorphic to $D\times\Cantor$ with the subspace identity representation: \Cref{lem:encoding} recovers $w$, and approximation within $1/4$ distinguishes each promised bit. Conversely $w,b$ compute all coordinates of the table. We use these isomorphisms whenever writing $\widehat F_j$ or $\widehat G_j$.

\begin{remark}[What remains fixed, and what changes]
The evaluation formula, value representations, output metric and output representation, and the base computation class $\Comp$ are fixed across the six packages. The allowed payload domain and its mathematical interpretation change. In particular, $\kappa(p)$ is an evaluation of the name-enriched input $(p,b)$, not of the unnamed function independently of its name. These are not three degrees of one fixed represented IVT relation. The example is not proposed as a natural numerical oracle; it isolates exactly the distinction between finite dependence on exact values and effective access through their names.
\end{remark}

\subsection{Common raw height, proved pointwise for all six problems}
\begin{theorem}[Raw classification of the constructed family; derived here]\label{thm:ivt-three}
Every member of
\[
 \mathcal U=\{P_{\mathrm A},Q_{\mathrm A},P_{\mathrm B},Q_{\mathrm B},P_{\mathrm C},Q_{\mathrm C}\}
\]
has $\SCImv=1$.
\end{theorem}
\begin{proof}[Upper bounds, including both general-algorithm axioms]
For $P_j$ define $r_j(p)=\min Z(h_j(p))$. For $Q_j$ define $s_j(q)=\min J_j(q)$. Each minimum exists by nonempty compactness. These are set-theoretic functions of the full payload, whether or not they are computable. The single exact atom $\kappa(w)$ determines $w$ and therefore determines the relevant minimum. Raw type~G imposes no regularity on this postprocessing.

The deepest maps are
\[
 \Gamma^H_{j,n}(p,b)=t_n(b)+\ell r_j(p),\qquad
 \Gamma^J_{j,n}(q,b)=t_n(b)+\ell s_j(q).
\]
Each uses the fixed, finite, nonempty query set $\{\lambda_*\}\cup\{\lambda_m:m<n\}$. Agreement on it determines the payload, the queried shift bits, and the output; the query set itself is fixed. Thus both axioms hold. Equation~\eqref{eq:shift-error} gives convergence, with error at most $3^{-n}/8$, to a permissible answer of the corresponding relation. All intermediate outputs lie in $[0,1]$.
\end{proof}
\begin{proof}[Height-zero obstruction for each member separately]
Fix the function $h_*(x)=x-1/2$ and the interval $J_*=[1/2,1/2]$. Each has a computable standard name; include advice $1/2$ in A and a constant array of the standard name in C. Fix these payloads. In each problem the only permissible answer on shift input $b$ is
\[
 t(b)+\ell/2.
\]
A hypothetical height-zero general selector asks finitely many evaluations at $(w_*,0^\omega)$. Choose an unqueried bit index $m$ and change just that bit to $1$. The payload atom and every queried bit stay the same, but the unique permissible answer changes by $\frac14\,3^{-(m+1)}>0$. General-algorithm output dependence is contradicted. This argument applies separately to each of the six members, not merely to a worst-case member.
\end{proof}

\section{Equivalences inside the transport hierarchy, with every trace supplied}\label{sec:traced-equivalences}
\begin{lemma}[One-atom trace compiler; derived here]\label{lem:atom-trace}
Consider two packages of the form \eqref{eq:payload-interface}. Let $\phi:D_S\to D_P$ be a computable partial Baire-space function, total on $D_S$, and suppose
\[
 F_P(\phi(w),b)\subseteq F_S(w,b)
\]
for every valid input. Then $E(w,b)=(\phi(w),b)$ and $D=\id_{[0,1]}$ witness $S\leq^{\mathrm{mv}}_{\mathrm{TTE},\mathrm{fq}}P$.
If $\phi$ is merely a set-theoretic function, the same conclusion holds with raw in place of TTE.
\end{lemma}
\begin{proof}
The evaluation-image isomorphisms above make $E$ computable. For the target atom use the one-entry source transcript $(\lambda_*)$ and the reconstruction
\begin{equation}\label{eq:atom-trace}
 \vartheta_*:z\longmapsto
         \kappa\bigl(\phi(\kappa^{-1}(z))\bigr),
       \qquad z\in\kappa(D_S).
\end{equation}
It has a partial computable implementation on a domain containing the transcript image. For target bit query $n$, use $(\lambda_n)$ and the identity reconstruction
\begin{equation}\label{eq:bit-trace}
 \vartheta_n(v)=v.
\end{equation}
A single program, given a target index, returns the appropriate one-query list and machine index. The decoder is total, computable and metrically continuous, and soundness is the assumed solution inclusion. This proves all parts of \Cref{def:mv-transport}. Without computability, the same formulas are admissible unrestricted raw reconstructions.
\end{proof}

\begin{theorem}[Three traced equivalences; constructions derived here]\label{thm:traced-pairs}
For each $j\in\{\mathrm A,\mathrm B,\mathrm C\}$,
\begin{equation}\label{eq:traced-pairs}
 P_j\equiv^{\mathrm{mv}}_{\mathrm{TTE},\mathrm{fq}}Q_j.
\end{equation}
All six witnessing transports use the identity output decoder and exactly one source evaluation per target evaluation.
\end{theorem}
\begin{proof}
The encodings preserve $b$ and replace the payload as follows:
\begin{center}
\small
\begin{tabular}{@{}lll@{}}
\toprule
Direction & Source payload & Target payload $\phi(w)$\\
\midrule
$P_{\mathrm A}\to Q_{\mathrm A}$ & $\pair{u,v}$ & $\pair{\operatorname{sgl}(v),v}$\\
$Q_{\mathrm A}\to P_{\mathrm A}$ & $\pair{q,v}$ & $\pair{\beta(q),v}$\\
$P_{\mathrm B}\to Q_{\mathrm B}$ & $u$ & $\alpha(u)$\\
$Q_{\mathrm B}\to P_{\mathrm B}$ & $q$ & $\beta(q)$\\
$P_{\mathrm C}\to Q_{\mathrm C}$ & $p$ & $\alpha^\Delta(p)$\\
$Q_{\mathrm C}\to P_{\mathrm C}$ & $q$ & $\beta^\Delta(q)$\\
\bottomrule
\end{tabular}
\end{center}
The arrow denotes the direction of reduction, not a claim of reverse computational dependence.

For the first row, the interval is the singleton containing the supplied zero; every target solution is therefore a source zero. Its advice remains valid. For the second, \eqref{eq:beta} makes the zero set exactly the source interval, and the supplied point is a valid zero. For the third row, \eqref{eq:alpha} gives a nonempty interval contained in the source zero set; the fourth has equality of those sets. Translating by $t(b)$ and scaling by $\ell$ preserve all inclusions.

For the fifth row, let $u=\Lim(p)\in D_H$. Equation~\eqref{eq:exact-lift} makes $\alpha^\Delta(p)$ a valid jumped interval name with limiting name $\alpha(u)$; its named interval is contained in $Z(\delta_H(u))$. The sixth row similarly has limiting name $\beta(\Lim(q))$ and a function with exactly the source interval as zero set. These statements check the entire target-domain promises, not only formal types of the arrays.

Every displayed payload transformation is computable. Apply \Cref{lem:atom-trace} in each row, with reconstructions \eqref{eq:atom-trace} and \eqref{eq:bit-trace}. Those are the promised finite traces. In particular we have not inferred \eqref{eq:traced-pairs} merely from a Weihrauch equivalence.
\end{proof}

\begin{remark}[What ``one query'' means here]
The reconstruction of the target atom may read arbitrarily many digits of a \textit{name of the one source real value} as output precision increases. The finite trace concerns the number of source evaluations, not a precision-independent bound on the number of bits read from their names. This is exactly the distinction built into the source trace definition. Replacing the atomic interface by individual payload-digit evaluations would require a new analysis; the one-query formula would no longer be a finite exact reconstruction of the target atom.
\end{remark}

\subsection{Forward TTE transports and reverse raw transports}
\begin{proposition}[Refining a raw equivalence class; derived here]\label{prop:raw-collapse}
There are transports
\begin{equation}\label{eq:TTE-chain}
 P_{\mathrm A}\leq^{\mathrm{mv}}_{\mathrm{TTE},\mathrm{fq}}P_{\mathrm B} \leq^{\mathrm{mv}}_{\mathrm{TTE},\mathrm{fq}}P_{\mathrm C},
\end{equation}
and all six members of $\mathcal U$ are raw-finite-query equivalent.
\end{proposition}
\begin{proof}
For the first forward arrow drop the root advice:
$\phi(\pair{u,v})=u$. The underlying function and the shift do not change. For the second use the constant array $\phi(u)(\pair{s,n})=u(n)$, which has limit $u$. \Cref{lem:atom-trace} provides their TTE traces.

For the reverse raw arrow $P_{\mathrm B}\leq_{\mathrm{raw},\mathrm{fq}}^{\mathrm{mv}}P_{\mathrm A}$, let $r(u)=\min Z(\delta_H(u))$ and choose a definite Cauchy name $\nu(r(u))$ (for example dyadic lower roundings at precision $2^{-(n+1)}$). Set $\phi(u)=\pair{u,\nu(r(u))}$. For $P_{\mathrm C}\leq_{\mathrm{raw},\mathrm{fq}}^{\mathrm{mv}}P_{\mathrm B}$ set $\phi(p)=\Lim(p)$. These are well-defined set-theoretic payload transformations on the promised domains. They preserve the semantic function and have raw one-atom reconstructions, without any claim of computability. \Cref{thm:traced-pairs} supplies raw equivalences between $P_j$ and $Q_j$. Composition proves the assertion for all six members.
\end{proof}

\section{The exact Weihrauch images and the three distinct transport degrees}\label{sec:ivt-degrees}

\begin{theorem}[Exact degrees of the transported presentations; derived here]\label{thm:ivt-degrees}
With the specified product-Cauchy information spaces,
\begin{align}
 \widehat F_{\mathrm A}\esW\widehat G_{\mathrm A}
       &\esW\id_{[0,1]},\label{eq:degree-A}\\
 \widehat F_{\mathrm B}\esW\widehat G_{\mathrm B}
       &\esW\CC,\label{eq:degree-B}\\
 \widehat F_{\mathrm C}\esW\widehat G_{\mathrm C}
       &\esW\CC'.\label{eq:degree-C}
\end{align}
Here $\id_{[0,1]}$ has Cauchy input and output. In ordinary Weihrauch degrees these are the strict chain
\[
 \one\LW\CC\LW\CC'.
\]
We use $\one$ only for the \textit{ordinary} computable degree; no identification of all computable strong degrees is intended.
\end{theorem}
\begin{proof}[Reduction to interval presentations]
\Cref{thm:traced-pairs} and the proved trace-forgetting implication give $\widehat F_j\esW\widehat G_j$. It remains to classify $\widehat G_j$; this step is a Weihrauch degree calculation, not another assertion about finite-query equivalence to an unspecified representative of a degree.
\end{proof}
\begin{proof}[Package A]
Given $q=\pair{q_0,v}$ and $b$, compute $r=\rho(v)$ and output $t(b)+\ell r$. This is a computable real-valued solution, so it strongly reduces to $\id_{[0,1]}$ by feeding the computed answer name to that identity problem.

Conversely, given a Cauchy name of $y\in[0,1]$, compute $r=(1+y)/3$, a Cauchy name $v$ of $r$, and the singleton name $\operatorname{sgl}(v)$. Use payload $\pair{\operatorname{sgl}(v),v}$ and $b=0^\omega$. Every output is the unique point
\begin{equation}\label{eq:affine-decoding}
 x=\frac14+\frac14\frac{1+y}{3}=\frac{4+y}{12}.
\end{equation}
The postprocessor $x\mapsto\max\{0,\min\{1,12x-4\}\}$ returns $y$ and does not need the input. The problem is pointed and computable, hence has ordinary degree $\one$.
\end{proof}
\begin{proof}[Package B: both strong reductions]
From a negative name $q$ of $J=[a,b_0]\in\mathcal J$ and the shift bits $b$, compute a negative name of
\[
 J^{\mathrm{sh}}=t(b)+\ell J\subseteq[0,1].
\]
For a completely explicit construction, take nested rational bounds $a_s\uparrow a$, $b_s\downarrow b_0$. Then
\[
 A_s=t_s(b)+\ell a_s,\qquad
 B_s=t_s(b)+\tfrac18\,3^{-s}+\ell b_s
\]
are nested lower and upper bounds for the endpoints of $J^{\mathrm{sh}}$. The lower bounds increase and the upper bounds decrease, because the truncated ternary sum plus its maximum possible tail decreases. Their limits are the correct endpoints. \Cref{lem:interval-name} turns these bounds into a negative interval name. Call $\CC$ on that name and output its answer unchanged. This is strong reducibility; the shift has been absorbed into the \textit{preprocessor}.

For the reverse reduction, given an arbitrary $J\subseteq[0,1]$, computably transform its negative name to a negative name $q$ of $(1+J)/3\in\mathcal J$. Set $b=0^\omega$ and supply $(\kappa(q),0,0,\ldots)$ as the evaluation table of $Q_{\mathrm B}$. Any answer is of the form \eqref{eq:affine-decoding} with $y\in J$. The same affine clamped decoder gives an answer to $\CC$ without the input. This proves \eqref{eq:degree-B}.
\end{proof}
\begin{proof}[Package C: valid jump-domain promises]
Let $q_s\to q_\infty=\Lim(q)$, where $q_\infty$ is a valid negative name of $J\in\mathcal J$. Let $\Sigma(q_\infty,b)$ denote the computable negative-name transformer for $t(b)+\ell J$ just constructed. Apply bounded-simulation lifting to $\Sigma$ using rows $(q_s,b)$, with the same $b$ in every row. This computably produces a jumped negative name whose limit is exactly $\Sigma(q_\infty,b)$. A $\CC'$ call therefore returns a point in $G_{\mathrm C}(q,b)$; the postprocessor is the identity. The intermediate rows need not be valid interval names.

For the reverse reduction, lift the computable negative-name transformer for $J\mapsto(1+J)/3$. Given a $\CC'$ input, this yields a valid jumped negative name $q$ of the contracted interval. Its Baire sequence $q$ is available computably from the original name, even though $\Lim(q)$ is not. Hence $\kappa(q)$ and the fixed zero shift form a computable $Q_{\mathrm C}$ input. Every answer decodes by $x\mapsto12x-4$ on the promised range to an element of the original interval. This proves \eqref{eq:degree-C}. The strict ordinary inequalities were proved in \Cref{thm:ivt-cc} and \Cref{prop:ivt-jump}.
\end{proof}

\begin{theorem}[Exact modal profile of the six-problem family; derived here using the quotient of {\cite[Definition 7.1]{SorgTransport}}]\label{thm:modal-profile}
For the explicitly defined multivalued extension,
\[
 \mathcal O^{\mathrm{mv}}_1(\mathcal U)=\mathcal U.
\]
The raw quotient consists of one class, whereas the TTE quotient consists of exactly the following three classes, in the displayed strict order:
\begin{equation}\label{eq:three-modal-degrees}
 [P_{\mathrm A}]_{\mathrm{TTE}}=[Q_{\mathrm A}]_{\mathrm{TTE}}\;<\;
 [P_{\mathrm B}]_{\mathrm{TTE}}=[Q_{\mathrm B}]_{\mathrm{TTE}}\;<\;
 [P_{\mathrm C}]_{\mathrm{TTE}}=[Q_{\mathrm C}]_{\mathrm{TTE}}.
\end{equation}
Forgetting traces maps these classes to the three distinct strong degrees in \Cref{thm:ivt-degrees}, and then to the ordinary degrees $\one,\CC,\CC'$. On this finite suborder the forgetful map is an order isomorphism onto its image; no global order-isomorphism claim is made.
\end{theorem}
\begin{proof}
The common raw height is \Cref{thm:ivt-three}. Raw collapse is \Cref{prop:raw-collapse}. The traced pair equivalences and the forward arrows \eqref{eq:TTE-chain} give at most three ordered TTE classes. A reverse TTE transport would, by \Cref{thm:mv-transport}, give the corresponding reverse ordinary Weihrauch reduction. The strict degree separations in \Cref{thm:ivt-degrees} rule out every such reverse arrow. Thus there are exactly three classes, with no additional identifications. The stated injectivity on this three-element family follows from those same separations.
\end{proof}

\begin{remark}[Three degrees are not three minimal exact sources]
The TTE order in \eqref{eq:three-modal-degrees} is a chain. Its unique minimal class in this ambient family is the A-class, not three incomparable minimal classes. All six members are individually exact at raw height one, which is stronger than merely saying that the family has worst-case height one. This is the family-pointwise distinction developed in \cite[Definitions 2.6-2.7, Lemma 2.9 and Proposition 2.10]{SorgWitness}, now verified directly for our family. We do not infer pointwise exactness merely from one hard member.
\end{remark}

\section{The pure implementations and the limits of proof translation}\label{sec:ivt-implementation}
\begin{theorem}[Explicit compilers and sharp pure heights; derived here using the normal form {\cite[Theorem 4.27]{SorgBridge}}]\label{thm:ivt-pure}
For $R_j=P_j$ or $Q_j$, respectively,
\[
 \puremv(R_{\mathrm A})=0,\qquad
 \puremv(R_{\mathrm B})=1,\qquad
 \puremv(R_{\mathrm C})=2.
\]
In each case one computable functional produces the entire array required by the pure normal form, and all successive limit-domain promises hold.
\end{theorem}
\begin{proof}
For A, decode the supplied root/point name and $b$, and directly compute a Cauchy name of $t(b)+\ell r$. This is a zero-limit realizer.

For $Q_{\mathrm B}$, compute the rational sequence $x_s=t_s(b)+\ell a_s$ from the nested interval bounds. It converges to $t(b)+\ell\min J$, a permissible answer. \Cref{lem:real-limit} produces a computable array of rational codes whose columns eventually stabilize to a Cauchy name of that answer. Let $B_J(q,b)$ be this array generator. For $P_{\mathrm B}$ set $B_H(u,b)=B_J(\alpha(u),b)$. The interval chosen by $\alpha$ is contained in the source zero set, so the final answer is again valid. Both $B_J$ and $B_H$ are single partial computable Baire-space functionals, total on all promised input names.

For C, write $w_s\to w_\infty=\Lim(w)$ and use the relevant function $B_J$ or $B_H$. Bounded-simulation lifting, with constant shift rows $b$, produces one computable array $K_C(w,b)$ satisfying
\[
 \Lim(K_C(w,b))=B_{J/H}(w_\infty,b).
\]
The first limit exists coordinatewise because each convergent simulation eventually follows the finite computation on the true limiting name. The right-hand side is in $\dom(\Lim)$ by the B-case proof. Thus
\[
 \rho\bigl(\Lim^{(2)}(K_C(w,b))\bigr) \in G_{\mathrm C}(w,b)\quad\hbox{or}\quad F_{\mathrm C}(w,b),
\]
as applicable. Compose with the computable evaluation-table decoders to obtain the required preprocessors on the actual product-Cauchy information spaces.

This proves all upper bounds with uniform witnesses. To prove sharpness, use the degree identifications of \Cref{thm:ivt-degrees}; \Cref{thm:ivt-cc} shows/recalls then that $\mathrm{CC}_{[0,1]}$ is noncomputable, which excludes pure height zero for either $\mathrm{B}$-presentation. \Cref{prop:ivt-jump} shows that $\mathrm{CC}_{[0,1]}'\not\leq_{\mathrm W}\Lim$, which excludes pure height at most one for either $C$-presentation.
Together with the explicit upper bounds, this gives the exact pure heights $0,1,2$; \eqref{eq:mv-bridge} identifies these heights with the corresponding limit-chain ranks.
\end{proof}

To summarize the implementation of the six transport presentations here:\\
The source domains, semantic decoders, and every jump promise are fixed in \eqref{eq:Pj}-\eqref{eq:Qj}. The interface is countable; its value representations and its product representation are ordinary Cauchy representations, not computationally atomic reals. The output is a point with a standard Cauchy name, and $\Comp$ is fixed throughout. \Cref{lem:atom-trace} proves uniformity in the target evaluation index, while \Cref{thm:ivt-pure} proves the separate tower-index uniformity. All transport decoders are total continuous identity maps. Exact degree lower bounds are established independently and are not inferred from raw height. The multivalued raw and pure definitions and the transport inclusion are the explicit local extensions, not source quotations.

\begin{remark}[The raw height-one proof is not itself a computable height-one compiler]\label{rem:no-mechanical}
The raw estimators recover $\min Z(h_j(p))$ or $\min J_j(p)$ by an arbitrary operation on the exact atom $\kappa(p)$. Its formal admissibility in type~G is not a computability certificate. The displayed numerical error $3^{-n}/8$ only controls the shift tail; it says nothing about effectiveness of the recovered minimum. In B the implemented proof replaces this arbitrary operation by a one-limit interval construction; in C two name-level limits are necessary. Consequently an identical raw height and even an explicit raw convergence rate do not determine implemented height.

The successful transport calculation is stronger than merely recording those numerical ranks: it supplies concrete equivalent interval presentations and finite traces, then identifies their exact degree images. But it does not establish an algorithm translating every raw SCI proof into a same-height pure proof, nor does it make an exact Weihrauch degree a function of the raw SCI number.
\end{remark}

\section{A literal single-valued companion}\label{sec:literal}
The multivalued extension was necessary for unrestricted zero selection. It is not necessary for the underlying counterexample to a rank-to-degree rule. 

The following family fits the single-valued SCI definitions of all three works \cite{SorgBridge, SorgWitness, SorgTransport} literally; it is a unique-root \textit{subproblem} of IVT, not a new degree claim about unrestricted IVT.

For $k\in\{0,1,2\}$ let
\[
 D_k^*=\{p\in\dom(\rho\circ\Lim^{(k)}): r_k(p):=\rho(\Lim^{(k)}(p))\in[1/3,2/3]\},
 \qquad \Omega_k^*=D_k^*\times\Cantor.
\]
Use the interface \eqref{eq:payload-interface}, its product-Cauchy information representation, standard real output, and the single-valued target
\[
 \Xi_k(p,b)=t(b)+\ell r_k(p).
\]
It is the unique zero of $T_bh_{r_k(p)}$, where $h_r(x)=x-r\in\mathcal H$.

\begin{theorem}[Unique-root separation under the original single-valued definition; derived here]\label{thm:literal}
For each $k\in\{0,1,2\}$,
\[
 \SCIG(P_k^*)=1,\qquad \widehat\Xi_k\eW\Lim^{(k)}, \qquad \pure(P_k^*)=k.
\]
In particular the three ordinary degrees are $\one$, $\Lim$, and $\Lim^{(2)}$.
\end{theorem}
\begin{proof}
The raw tower is $\Gamma_{k,n}(p,b)=t_n(b)+\ell r_k(p)$. The exact atom determines the set-theoretic value $r_k(p)$; the remaining queries are the first $n$ shift bits. The two general-algorithm axioms and the error bound are the same as in \Cref{thm:ivt-three}. Fixing a payload for $r=1/2$ and changing an unqueried bit proves the height-zero obstruction. This time the target is single-valued, so the conclusion is literally about $\SCIG$.

For the upper degree reduction decode $p,b$, use $\Lim^{(k)}$ to obtain a Cauchy name of $r_k(p)$, and compute $t(b)+\ell r_k(p)$ with the original input available to the ordinary postprocessor.

For the lower reduction use the computable embedding
\[
 e:\B\to[1/3,2/3],\qquad e(q)=\tfrac13+\tfrac13\kappa(q).
\]
Its inverse is computable on its range. Let $L$ be a computable name transformer sending $q$ to a Cauchy name of $e(q)$. Iterated bounded-simulation lifting produces a computable $L^{\Delta k}$ such that
\[
 \Lim^{(k)}(L^{\Delta k}(P))=L(\Lim^{(k)}(P))
\]
for $P\in\dom(\Lim^{(k)})$. Apply this to a $\Lim^{(k)}$ instance $P$, set $u=L^{\Delta k}(P)$, and use the SCI input with evaluation table $(\kappa(u),0,0,\ldots)$. Its output is
\[
 x=\tfrac14+\tfrac14e(q),\qquad q=\Lim^{(k)}(P).
\]
From $x$ recover $e(q)=4x-1$ and then $q=\kappa^{-1}(3e(q)-1)$. These postprocessing operations are computable on the promised range, independently of the original input, so $\Lim^{(k)}\sW\widehat\Xi_k$.

Finally $\one\LW\Lim$ by \Cref{lem:plpo}; and $\Lim\LW\Lim^{(2)}$, since otherwise $\INF\lW\Lim^{(2)}$ and $\INF\nW\Lim$ from \Cref{thm:inf-degree} would contradict transitivity. The literal single-valued bridge gives the asserted pure heights.
\end{proof}

\textbf{No unsupported promotion to a traced equivalence.}
The last proof establishes Weihrauch equivalence to the bare operators $\Lim^{(k)}$. It does not endow those operators with arbitrary SCI interfaces and then claim finite-query equivalence. Such an additional claim would require declared benchmark interfaces and traces, exactly as supplied for $P_j,Q_j$ in \Cref{thm:traced-pairs}. Keeping this distinction visible is part of the purpose of \Cref{part:transport}.

\end{document}